\documentclass[11pt]{amsart} 

\usepackage{latexsym}
\usepackage{amssymb}
\usepackage{amsmath}
\usepackage{amscd}

\theoremstyle{plain}
\newtheorem{theo}{Theorem}[section]
\newtheorem{lema}[theo]{Lemma}
\newtheorem{prop}[theo]{Proposition}
\newtheorem{cor}[theo]{Corollary}

\theoremstyle{definition}

\newtheorem{rem}{Remark}[section]

\newtheorem{defn}{Definition}[section]

\DeclareMathOperator{\Aut}{{\rm Aut}}

\DeclareMathOperator{\GL}{{\rm GL}}
\DeclareMathOperator{\SL}{{\rm SL}}
\DeclareMathOperator{\M}{{\rm M}}            

\DeclareMathOperator{\nor}{{\rm {\rm N}_G}}

\DeclareMathOperator{\Zcl}{{\rm Zcl}}        
\DeclareMathOperator{\supp}{{\rm supp}}      
\DeclareMathOperator{\Ad}{{\rm Ad}}          
\DeclareMathOperator{\rad}{{\rm rad}}
\DeclareMathOperator{\Lie}{{\rm Lie}}        

\newcommand{\vect}[1]{{\boldsymbol{#1}}}
\newcommand{\ve}{\vect{e}}
\newcommand{\vp}{\vect{p}}
\newcommand{\vs}{\vect{s}}
\newcommand{\vt}{\vect{t}}
\newcommand{\vv}{\vect{v}}
\newcommand{\vy}{\vect{y}}

\newcommand{\field}[1]{\mathbb{#1}} 
\newcommand{\C}{\field{C}} 
\newcommand{\R}{\field{R}}
\newcommand{\N}{\field{N}}
\newcommand{\Q}{\field{Q}}
\newcommand{\Z}{\field{Z}}

\providecommand{\abs}[1]{\lvert#1\rvert}

\providecommand{\pline}[1]{<\!#1\!>}
\providecommand{\pLine}[1]{\bigl<#1\bigr>}
\providecommand{\trn}[1]{{\,^{\bf T}\!#1}}
\providecommand{\card}[1]{\#(#1)}

\newcommand{\num}{\abs}

\renewcommand{\setminus}{\smallsetminus}

\newcommand{\qp}{K}   
\newcommand{\zp}{O}   
\newcommand{\frakO}{{\mathfrak O}} 

\newcommand{\inv}{^{-1}}

\newcommand{\cl}[1]{\overline{#1}}

\newcommand{\smat}[1]{\bigl(\begin{smallmatrix} #1 \end{smallmatrix}\bigr)}

\def\frakF{\mathfrak{F}}

\newcommand{\cB}{\mathcal{B}}
\newcommand{\frakB}{\mathfrak{B}}
\newcommand{\cP}{\mathcal{P}}
\newcommand{\cJ}{\mathcal{J}}
\newcommand{\cC}{\mathcal{C}}
\newcommand{\cM}{\mathcal{M}}
\newcommand{\cG}{\mathcal{G}}
\newcommand{\cR}{\mathcal{R}}
\newcommand{\cS}{\mathcal{S}}
\newcommand{\cT}{\mathcal{T}}
\newcommand{\cH}{\mathcal{H}}
\newcommand{\cW}{\mathcal{W}}

\newcommand{\frakV}{\mathfrak{V}}

\begin{document}

\title{Unipotent flows on products of   ${\rm SL}(2,K)/\Gamma$'s}
\author{Nimish A. Shah}
\address{Tata Institute of Fundamental Research, Mumbai 40005, India}
\email{nimish@math.tifr.res.in}

\date{}

\begin{abstract}
  We will give a simplified and a direct proof of a special case of
  Ratner's theorem on closures and uniform distribution of individual
  orbits of unipotent flows; namely, the case of orbits of the
  diagonally embedded unipotent subgroup acting on ${\rm
    SL}(2,K)/\Gamma_1\times \dots \times {\rm SL}(2,K)/\Gamma_n$,
  where $K$ is a locally compact field of characteristic $0$ and each
  $\Gamma_i$ is a cocompact discrete subgroup of ${\rm SL}(2,K)$.
  This special case of Ratner's theorem plays a crucial role in the
  proofs of uniform distribution of Heegner points by Vatsal, and
  Mazur conjecture on Heegner points by C.  Cornut; and their
  generalizations in their joint work on CM-points and quaternion
  algebras. A purpose of the article is to make the ergodic theoretic
  results accessible to a wide audience.
\end{abstract}

\maketitle

\section{Introduction}

In the mid seventies M.S. Raghunathan had conjectured that dynamical
properties of individual orbits of unipotent flows on finite volume
homogeneous spaces of semisimple Lie groups show a remarkable
algebraic behaviour; namely, the closure of any non-periodic orbit is
a finite volume homogeneous space of a larger subgroup. This
conjecture was motivated by an approach to resolve Oppenheim
conjecture on values of quadratic forms at integral points. A precise
form of Raghunathan's conjecture, and its important measure theoretic
analogues were formulated by S.G. Dani, who also verified those
conjectures for horospherical flows in the early eighties. This work
attracted greater attention to the Raghunathan conjecture and its
extensions.  It generated a lot of excitement when in the late
eighties G.A. Margulis fully settled the Oppenheim conjecture in
affirmation by verifying Raghunathan's conjecture for certain very
specific cases.  This seems to be the first major triumph of the power
of ergodic theoretic methods in solving long standing number theoretic
problems.  Soon after, by the beginning of the nineties M. Ratner
obtained complete affirmative resolution of the above mentioned
conjectures on unipotent flows, and also proved the uniform
distribution for the individual orbits, through a series of long
technical papers~\cite{R:invent,R:acta,R:measure,R:duke} involving
many deep ideas.  Ratner's theorems were very powerful tools ready to
be used.  Since than several types of new Diophantine approximation
results have been proved using the algebraic properties of unipotent
dynamics. The dynamical results were later generalized for $p$-adic
Lie groups by Ratner~\cite{R:p-adic}; as well as by Margulis and
Tomanov~\cite{Mar+Tom:Invent}, whose also gave shorter and more
conceptual proofs in all cases.

What really surprises me about the $p$-adic case of Ratner theorem is
the way it gets utilized in the work of V.
Vatsal~\cite{Vatsal:uniform} on uniform distribution of Heegner
points. Using a combination of remarkable number theoretic results and
his observations, Vatsal reduced the study of distribution of Heegner
points to the following combinatorial problem:

Let $\cT$ be a $p+1$-regular tree for a prime $p$, and
$\cG=\cT/\Gamma$ be a finite quotient graph, where $\Gamma$ is group
of automorphisms of $\cT$ with finite stabilizers of vertices. Let
$\Gamma'$ be a conjugate of $\Gamma$ in $\Aut(\cT)$ such that $\Gamma$
and $\Gamma'$ do not have a common subgroup of finite index; that is,
they are not commensurable. Fix a base point $v_0$ in $\cT$, and let
$\cT(n)$ denote the vertices of $\cT$ at the distance $n$ from $v_0$.
Consider the finite graph $\cG'=T/\Gamma'$, and let $q:\cT\to \cG$ and
$q':\cT\to \cG'$ denote the natural quotient maps. We embed $\cT$
diagonally in $\cT\times \cT$, and project it onto $\cG\times \cG'$;
more precisely we consider the map $\bar\Delta:\cT\to \cG\times \cG'$
given by $\bar\Delta(v)=(q(v),q'(v))$. The question is whether
$\bar\Delta(T(n))$ surjects onto $\cG\times \cG'$ for large $n$, and
does it visit all points of the product graph with the correct
limiting frequency as $n\to\infty$?

His question was motivated by the fact that on a finite non-bipartite
regular graph, a random walk of step $n$ is uniformly distributed as
$n\to\infty$. On the other hand in this case it is already a question
whether the image of the diagonally embedded $\cT$ is surjective on
$\cG\times \cG'$. In the actual situation of interest, $\cT\cong
\SL_2(\Z_p)\backslash\SL_2(\Q_p)$, realized as the Bruhat-Tits tree,
and $\Gamma$ is a cocompact discrete subgroup of $\SL_2(\Q_p)$ so that
$\cG$ is associated to the quotient by the right action of $\Gamma$,
and $\Gamma'$ is a conjugate of $\Gamma$ in $\SL_2(\Q_p)$. Therefore
the surjectivity of the diagonal embedding follows if we can show that
the set $\Gamma\Gamma'$ is dense in $\SL_2(\Q_p)$; or more generally,
if the element-wise product of any two non-commensurable lattices in
$\SL_2(\Q_p)$ is a dense subset of $\SL_2(\Q_p)$.

Vatsal asked this question to Raghunathan, who realizing this as a
question about orbit closures for $\Gamma$-action on
$\SL_2(\Q_p)/\Gamma'$ consulted Dani. The same question was earlier
posed and answered in author's Masters thesis~\cite{Shah:MSc} for
lattices in $\SL_2(\R)$ and $\SL_2(\C)$, and later in
\cite{Shah:genuni} for the lattices in arbitrary real semisimple Lie
groups using Ratner's theorem. Dani informed Vatsal that his guess was
indeed correct, and showed how to deduce the density result using
orbit closure results for actions of semisimple subgroup on $p$-adic
homogeneous spaces. Later using Ratner's uniform distribution results
for unipotent flows on the homogeneous space $\SL_2(\Q_p)/\Gamma\times
\SL_2(\Q_p)/\Gamma'$, Vatsal also deduced the uniform distribution for
the set $\bar\Delta(\cT(n))$ as $n\to\infty$ in $\cG\times \cG'$.

It is remarkable that the above seemingly combinatorial question about
products of certain finite graphs turns out to be intimately connected
to deep algebraic behaviour of ergodic properties of unipotent flows;
and these flows are analysed using local arguments involving the
adjoint actions on the Lie algebra near the origin.

In what follows, we would like to give a self contained proof of the
above surjectivity of the diagonal embedding of a tree in the product
of several regular finite graphs as above. The published proofs of
Ratner's theorem for $p$-adic Lie groups are quite intricate and they
require taking care of many different possibilities associated to the
general case.  Our purpose here is to follow the original arguments of
Margulis~\cite{Mar:Oppen-Oslo} used in his proof of Oppenheim
conjecture, as well as those used in its extensions by Dani and
Margulis~\cite{Dani+Mar:Oppen-Invent}, along with additional
observations to give an elementary proof.

In later works \cite{Cor:Mazur,Vat:special,CV:cm}, Vatsal and Cornut
also require the closure and the uniform distribution results for
products of several copies of $\SL_2(\qp)$ for any finite extension
$\qp$ of $\Q_p$. To take care of this, we have given our proofs for
all local fields $K$ of characteristic $0$ in place of $\Q_p$, without
introducing any extra complications.

After the introduction, the article gets divided into two independent
parts. In \SS~2--4, a proof of the orbit closure result is given. Near
the end of this proof we also need to assume a technical result on
`uniform recurrence in linear time' on the `non-singular' set for the
case of the product of $n-1$-copies.  The \SS~6 to 9 are devoted to
proving this result, which in other words says that a non-singular
unipotent orbit contributes zero measure on the singular set in its
limiting distribution. Once we have proved this result, in \S~10 we
combine it with Ratner's description of ergodic invariant measures for
unipotent flows and quickly deduce the result on uniform distribution.
In this way, it is possible to directly proceed to \S~6, directly
after reading the Introduction, if one is only interested in the
uniform distribution result. The Section~5 in the middle is devoted to
results on closures of $H$-orbits and commensurability of lattices.

\subsection{Notation}
Let $\qp$ denote a local field of characteristic zero.  Let $n\geq 1$
be given. Let $G=\SL_2(\qp)^n$.  For $\emptyset\neq
J\subset\{1,\ldots,n\}$, let
\begin{eqnarray*}
G_J&=&\{(g_1,\ldots,g_n)\in G: g_k=e,\, \forall k\not\in J\} \\
H_J&=&\{(g_1,\ldots,g_n)\in G_J: g_i=g_j,\, \forall i,j\in J;\ g_k=e,\
\forall k\not\in J\}.
\end{eqnarray*}
Then $G_J\cong\SL_2(\qp)^{\num{J}}$ and $H_J\cong \SL_2(\qp)$, which is
diagonally embedded in $\SL_2(\qp)^{\num{J}}$, where $\num{J}$ denotes the
cardinality of $J$. If $\num{J}=1$ then $H_J=G_J$.

Let $\cC$ denote the collection of sets of the form
$\cJ=\{J_1,\ldots,J_m\}$, where $1\leq m\leq n$, $J_i\subset
\{1,\ldots,n\}$, $J_i\neq\emptyset$, and $J_i\cap J_j=\empty$ for all
$i\neq j$. Define
\begin{equation*}
H_{\cJ}=H_{J_1}\cdots H_{J_m}.
\end{equation*}

Let 
\begin{equation*}
w_1(t)=\smat{1&t\\0&1}, \  \forall t\in\qp; \qquad 
d_1(\alpha)=\smat{\alpha& \\ &\alpha\inv}, \  \forall
\alpha\in\qp^\times;
\end{equation*}
\begin{eqnarray*}
W&=&\{w(\vt)=(w_1(t_1),\ldots,w_1(t_n)):
    \vt=(t_1,\ldots,t_n)\in \qp^n\} \\
A&=&\{(d_1(\alpha_1),\ldots,d_1(\alpha_n)):\alpha_j\in\qp^\times\}. 
\end{eqnarray*} 
We also consider
\begin{eqnarray*}
H&=&\{(g,\ldots,g)\in G: g\in\SL_2(\qp)\}=H_{\{1,\ldots,n\}}\\ 
U&=&\{u(t)=w(t,\ldots,t):t\in\qp\}=W\cap H\\
D&=&\{d(\alpha)=(d_1(\alpha),\ldots,d_1(\alpha)):\alpha\in\qp^\times\}=A\cap
H, \\ 
U^\perp&=&\{w(t_1,\ldots,t_{n-1},0):t_j\in\qp\}.
\end{eqnarray*}

\subsubsection*{Assumption.}
For $j=1,\ldots,n$, let $\Gamma_j$ be a discrete subgroup of
$G_{\{j\}}$ such that $G_{\{j\}}/\Gamma_j$ is compact, and let
$\Gamma=\Gamma_1\cdots\Gamma_n$. Then
\begin{equation} \label{eq:Gamma-prod}
  G/\Gamma\cong G_{\{1\}}/\Gamma_1 \times \dots \times
  G_{\{n\}}/\Gamma_n.
\end{equation} 

In this article, we will consider the action of $G$ on $G/\Gamma$ by
left translations; that is, if $g\in G$ and $x\in G/\Gamma$ then
$gx:=(gg_1)[\Gamma]$, where $g_1\in G$ is such that $x=g_1[\Gamma]$ is
the coset of $g_1$ in $G/\Gamma$. Also for any $A\subset G$ and
$X\subset G/\Gamma$, we define $AX=\{ax:a\in A,\,x\in X\}\subset
G/\Gamma$. 

We endow $G/\Gamma$ with the quotient topology; that is, a set
$X\subset G/\Gamma$ is closed (or open) if and only if its inverse
image in $G$ is closed (resp.\ open). Thus, given any $A\subset G$,
and $x=g[\Gamma]\in G/\Gamma$ for some $g\in G$, the set $Ax$ is
closed in $G/\Gamma$ if and only if $Ag\Gamma$ is a closed subset of
$G$.

Let
\begin{equation*}
\cC_0:=\{\cJ:\cup_{J\in\cJ} J
=\{1,\ldots,n\}\}=\{\cJ\in\cC:H_{\cJ}\supset U\}.
\end{equation*}

\subsection{Statements of the main results}

\begin{theo} \label{thm:closure}
  Given $n\geq 1$, let $G$, $\Gamma$, $U$, and the other notation be
  as above. For any $x\in G/\Gamma$, there exists $\cJ\in\cC_0$ and
  $w\in W$ such that
\begin{equation*}
\cl{Ux}=(wH_{\cJ} w\inv)x.
\end{equation*}
\end{theo}

\begin{defn} \label{def:multi-para} \rm
  A {\em multi-parameter\/} subgroup of $W$ is a subgroup of $W$ of
  the form $V=\{w(\vt):\vt\in\frakV\}$, where $\frakV$ is a subspace
  of $\qp^n$. We define $\dim V:=\dim_\qp(\frakV)$.
\end{defn}

\begin{cor} \label{cor:multi-para}
  Given $n\geq 1$, let $G$, $\Gamma$, and the other notation be as
  above. Let $V$ be a multi-parameter subgroup of $W$. Then for any
  $x\in G/\Gamma$, there exists $\cJ\in\cC$ and $w\in W$ such that
  $\cl{Vx}=wH_{\cJ}w\inv x$.
\end{cor}

\begin{cor} \label{cor:H}
  For any $x\in G/\Gamma$, there exists $\cJ\in\cC_0$ such that
\begin{equation*}
\cl{Hx}=H_{\cJ} x.
\end{equation*}
\end{cor}

In order to describe the relation between $H$, $\Gamma_i$'s, and
$\cJ$, we need some definitions.

In a topological group, two infinite discrete subgroups $\Lambda$ and
$\Lambda'$ are said to be {\em commensurable}, if $\Lambda\cap
\Lambda'$ is a subgroup of finite index in both, $\Lambda$ and
$\Lambda'$.

For $i=1,\dots,n$, let $p_i:G\to\SL_2(K)$ denote the projection on the
$i$-the factor. Let $x_0=e\Gamma$ denote the coset of the identity in
$G/\Gamma$. 

\begin{prop} \label{prop:H:comm} Suppose that $H_{\cJ}x_0$ is compact
  for some $\cJ\in\cC_0$. Then for any $J\in\cJ$ and any $i,j\in J$,
  the lattices $p_i(\Gamma_i)$ and $p_j(\Gamma_j)$ in $\SL_2(K)$ are
  commensurable.
\end{prop}

Combining this fact with Corollary~\ref{cor:H} immediately gives the
next result. Note that $H_\cJ=G$ if and only if $\cJ=\{{1},\dots,\{n\}\}$.

\begin{cor} \label{cor:H:noncomm} If $p_i(\Gamma_i)$ and
  $p_i(\Gamma_j)$ are not commensurable for all $i\neq j$ then
  $\cl{H\Gamma}=G$. \qed
\end{cor}

More generally, we will show the following:

\begin{cor} \label{cor:H:closure} Let $\cJ\in\cC_0$ be the partition
  of $\{1,\dots,n\}$ such that for any $i,j$, we have $i,j\in J$ for
  some $J\in\cJ$ if and only if $p_i(\Gamma_i)$ and $p_i(\Gamma_j)$
  are commensurable. Then $\cl{Hx_0}=H_{\cJ}x_0$.
\end{cor}

\subsection{Singular set for the $U$-action}  
\label{subsec:singular}
In the proof of Theorem~\ref{thm:closure} we will need to understand
the set of points for which the closure of the $U$-orbit is contained
in a closed orbit of a strictly lower dimensional subgroup of $G$.

More precisely, we say that a point $x\in G/\Gamma$ is {\em
  singular\/} (for the $U$-action on $G/\Gamma$) if $Ux\subset
(wH_{\cJ} w\inv)x$ and $(wH_{\cJ} w\inv) x$ is compact for some
$\cJ\in\cC_0$ and $w\in U^\perp$, such that $H_\cJ\neq G$. 

The set of singular points (for the $U$-action on $G/\Gamma$) is
denoted by $\cS(U,\Gamma)$. 

Note that if $n=1$ then $\cS(U,\Gamma)=\emptyset$.

\begin{prop} \label{prop:generic}
  There always exists a non-singular point for the $U$-action on
  $G/\Gamma$; that is $G/\Gamma\neq \cS(U,\Gamma)$.
\end{prop}

This fact can be proved quickly as follows: There exists a unique
$G$-invariant probability measure $\nu$ on $G/\Gamma$; that is,
$\nu(gE)=\nu(E)$ for any measurable set $E\subset G/\Gamma$ and any
$g\in G$. By Moore's ergodicity theorem, $U$-acts ergodically on
$G/\Gamma$ with respect $\nu$. Since $\nu(E)>0$ for any nonempty open
subset of $G/\Gamma$, by Hedlund's lemma, $\cl{Uy}=G/\Gamma$ for
$\nu$-almost all $y\in G/\Gamma$. Hence $\nu(\cS(U,\Gamma))=0$.

In subsection~\ref{subsec:generic} we will also give a simple proof of
Proposition~\ref{prop:generic} (without using Moore's ergodicity) by
showing that $\cS(U,\Gamma)$ is the image of a union of countably many
algebraic subvarieties of $G$ of strictly lower dimension.

As mentioned before following property of unipotent flows, called {\em
  uniform recurrence in linear time}\/ in \cite{Dani+Mar:Oppen-Invent}, at
the end of the proof of Theorem~\ref{thm:closure}.

\begin{theo} \label{thm:avoid2} Let $x_i\to x$ be a sequence in
  $G/\Gamma$ such that $x\not\in \cS(U,\Gamma)$. Then for any sequence
  $t_i\to\infty$ in $\qp$ and a compact neighbourhood $\frakO$ of $0$
  in $\qp$, there exists $t_i'\in (1+\frakO)t_i$ for every $i\in\N$,
  such that, after passing to a subsequence, $u(t_i')x_i\to y$ for
  some $y\in G/\Gamma\setminus\cS(U,\Gamma)$.
\end{theo} 

Note that if $G=\SL_2(\qp)$; that is $n=1$, then
Theorem~\ref{thm:avoid2} is a triviality, because
$\cS(U,\Gamma)=\emptyset$ in this case.

Moreover for proving the Theorem~\ref{thm:closure} for any given $n$,
we will need to use Theorem~\ref{thm:avoid2} only for
$G=\SL_2(\qp)^m$, where $m<n$. 

Therefore the proof of Theorem~\ref{thm:closure} for $n=2$ uses only
the trivial case of Theorem~\ref{thm:avoid2}; that is for $n=1$.

The Theorem~\ref{thm:avoid2} is actually derived as a consequence of a
more general result about limiting distribution of a sequence of
$U$-trajectories on the singular set. Since the techniques of proving
this result are very different from the remaining part of the proof of
Theorem~\ref{thm:closure} we have included all those results in a
second part of this article. In the second part of this article we
will also prove the uniform distribution result assuming Ratner's
description of ergodic $U$-invariant measures. In fact, the first part
of this article uses some of the ideas which have their analogues in
the classification of ergodic invariant measures for the $U$-action.

\section{Preliminaries}

\subsection{A result in ergodic theory}

We recall a result from \cite[Prop.~1.5]{Dani+Raghavan:frames}
\begin{prop} \label{prop:horo}
  For any $x\in G/\Gamma$, the orbit $DWx$ is dense in $G/\Gamma$.
\end{prop}

\begin{proof} 
  Take any $\alpha\in\qp$ such that $\abs{\alpha}_p>1$ and let
  $a=d(\alpha)$. By Mautner's Phenomenon
  (see~\cite{Mautner,Bekka:book}), $a$ acts ergodically on $G/\Gamma$.
  Therefore by Hedlund's lemma there exists $y\in G/\Gamma$ such that
  \begin{equation} \label{eq:dense}
    \cl{\{a^i:i>0\}y}=G/\Gamma.
  \end{equation}

  Let a sequence $\{y_k\}\in \{a^i:i>0\}y$ be such that $y_k\to x$ as
  $k\to\infty$.
  
  Let $z\in G/\Gamma$ be given. Then by \eqref{eq:dense} there exists a
  sequence $i_k\to\infty$ such that $a^{i_k}y_k\to z$, as $k\to\infty$.
  
  Let a sequence $g_k\to e$ in $G$ be such that $y_k=g_kx$ for all $k$.
  Since $\Lie(G)=\trn{[\Lie(W)]}\oplus \Lie(A)\oplus\Lie(W)$, there
  exist sequences $\vs_k\to 0$ and $\vt_k\to 0$ in $\qp^n$, $d_k\to
  e$ in $A$, and a $k_0\in\N$ such that
  \begin{equation*}
    g_k=\trn{[w(\vs_k)]}d_kw(\vt_k), \qquad \forall k\geq k_0.
  \end{equation*}
  
  Therefore $a^{i_k}y_k=(a^{i_k}g_ka^{i_k})a^{i_k}y_k$. Now
  \begin{equation*}
    a^{i_k}g_ka^{-i_k}=\trn{[w(\alpha^{-2i_k}\vs_k)]}d_kw(\alpha^{2i_k}\vt_k)
    =: \delta_k w_k,\qquad \forall k\geq k_0,
  \end{equation*}
  where $\delta_k:=\trn{[w(\alpha^{-2i_k}\vs_k)]}d_k$ and
  $w_k:=w(\alpha^{2i_k}\vt_k)\in W$. Thus
  $a^{i_k}y_k=\delta_kw_ka^{i_k}x$ and $\delta_k\to e$. Therefore
  \begin{equation*}
    w_k a^{i_k}x=\delta_k\inv (a^{i_k}y_k)\to\lim_{k\to\infty} a^{i_k}y_k=z.
  \end{equation*}
  Thus $z\in \cl{WDx}=\cl{DWx}$. This shows that $G/\Gamma\subset
  \cl{DWx}$.
\end{proof}

The proofs of Mautner's phenomenon and Hedlund's lemma are very nice
and short \cite{Bekka:book}. The above result deviates from the
classical ergodic theory results in one essential way; namely it tells
something about the dynamical property of each individual orbit,
rather than of almost every orbit. It is due to this reason we are
able use the above result for problems in number theory.

\subsection{Basic lemmas on minimal sets for group actions}

In this subsection let $G$ be a locally compact second countable
topological group acting continuously on a topological space $\Omega$.
For a subgroup $F$ of $G$, a subset $X$ of $\Omega$ is called {\em
  $F$-minimal\/} if $X$ is closed, $F$-invariant, and does not contain
any proper closed $F$-invariant subset. Thus if $X$ is $F$-minimal
then $\cl{Fx}=X$ for every $x\in X$. By Zorn's lemma, any compact
$F$-invariant subset of $\Omega$ contains an $F$-minimal subset.

\begin{lema}[Margulis~\cite{Mar:Oppen-simple}] 
  \label{lema:top1}
  Let $F$, $P$ and $P'$ be subgroups of $G$ such that $F\subset P\cap
  P'$.  Let $Y$, $Y'$ be closed subsets of $\Omega$, and $M\subset G$ be
  any set. Suppose that
  \begin{enumerate}
  \item $PY\subset Y$, $P'Y'\subset Y'$,
  \item $mY\cap Y'\neq\emptyset$ for all $m\in M$, and
  \item $Y$ is compact and $F$-minimal.
  \end{enumerate}
  Then $gY\subset Y'$ for all $g\in \nor(F)\cap\cl{P'MP}$.
  
  In particular, if $Y'=Y$ then $Y$ is invariant under the closed
  subgroup generated by $\nor(F)\cap\cl{P'MP}$.
\end{lema}

\begin{proof} 
  Let $g\in \cl{P'MP}$. There exist sequences $\{p'_i\}\subset
  P'$, $\{m_i\}\subset M$, and $\{p_i\}\subset P$ such that
  $p'_im_ip_i\to g$ as $i\to\infty$.
  
  By 2), for each $m_i$ there exists a $y_i\in Y$ such that $m_iy_i\in
  Y'$.  Since $\{p_i\inv y_i\}\subset Y$ and $Y$ is compact,
  by passing to subsequences, we may assume that $p_i\inv y_i\to y$ for
  some $y\in Y$.  Now $\{p'_im_iy_i\}\subset Y'$. Therefore as
  $i\to\infty$,
  \begin{equation*}
    p'_im_iy_i=(p'_im_ip_i)(p_i\inv y_i) \to gy\in Y'.
  \end{equation*}
  
  Further if $g\in N_G(F)$, then
  \begin{equation*}
    Y'\supset \cl{Fgy}=\cl{gFy}=g\cl{Fy}=gY,
  \end{equation*}
  where $\cl{Fy}=Y$ because $Y$ is $F$-minimal.
\end{proof}

\begin{lema}[Margulis~\cite{Mar:Oppen-simple}]
  \label{lema:top2}
  Assume that $G$ acts transitively on $\Omega$. Let $F$ and $P$,
  where $F\subset P$, be a closed subgroups of $G$, and $Y$ be a
  compact $F$-minimal subset of $\Omega$. Suppose there exists $y\in
  Y$ and a neighbourhood $\Phi$ of the identity in $G$ such that
  \begin{equation} 
    \label{eq:open} 
    \{g\in\Phi:gy\in Y\}\subset P.
  \end{equation} 
  Then $\cl{\eta(F)}$ is compact in $P/P_y$, where $P_y=\{g\in P:gy=y\}$
  and $\eta:P\to P/P_y$ is the natural quotient map.
\end{lema}

\begin{proof} 
  It is enough to show that given a sequence $\{f_i\}\subset F$,
  the sequence $\{\eta(f_i)\}$ has a convergent subsequence.
  
  To show this, we note that after passing through a subsequence,
  $f_iy\to z$ for some $z\in Y$. Since $\Omega$ is a homogeneous space
  of $G$, $\Phi y$ is a neighbourhood of $y$ in $\Omega$. Now since
  $Y$-is $F$-minimal, $Fy$ is dense in $Y$, and hence there exists $f\in
  F$ such that $fz\in \Phi y$. Therefore by \eqref{eq:open}, $fz=p'y$
  for some $p'\in P$. Hence $z=py$, where $p=f\inv p'\in P$. Thus
  $f_iy\to py$. Hence $(p\inv f_i)y\to y$.
  
  Again by \eqref{eq:open} there exists a sequence $p_i\to e$ in $P$
  such that $(p\inv f_i)y=p_iy$ for all large $i$. Thus $f_iy=pp_iy$;
  and hence $f_i\inv pp_i\in P_y$ for all large $i$. Therefore
  $\eta(f_i)=\eta(pp_i)\to \eta(p)$ as $i\to\infty$.
\end{proof}

\subsection{Limit set of a sequence of unipotent trajectories on a 
vector space}

Later after applying Lemma~\ref{lema:top1}, we will proceed further
using the following result.

\begin{prop}
  \label{prop:UMU}
  Let $M\subset G\setminus \nor(U)$ such that $e\in \cl{M}$. Then the
  closure of the subgroup generated by $\cl{UMU}\cap \nor(U)$ contains
  either $wDw\inv$ for some $w\in W$, or a nontrivial one-parameter
  subgroup of $U^\perp$.
\end{prop}

The proof of this proposition is based on the following general result
\cite{Mar:Oppen-Oslo,Dani+Mar:Oppen-Invent}: Let $V$ be a finite
dimensional vector space over $\qp$ and $U=\{u(t)\}_{t\in\qp}$ be a
nontrivial one-parameter unipotent subgroup of $\GL(V)$ and $\{p_i\}$
be a sequence of points in $V$ such that each of the trajectories
$\{u(t)p_i\}_{t\in\qp}$ is non-constant. Let $L$ denote the space of
$U$-fixed vectors in $V$. Now if $p_i\to p$ for some $p\in L$ then,
after passing to a subsequence, the following holds: there exist a
sequence $t_i\to\infty$ in $\qp$ and a non-constant polynomial map
$\phi:\qp\to V$ such that for any $s\in\qp$, we have $u(st_i)p_i\to
\phi(s)$ as $i\to\infty$.

We will prove this only for the cases needed for our purpose. 

Let $V=\qp^2$ and consider the standard linear action of $\{w_1(t)\}$
on $\qp^2$.  Let $I_0=\smat{1\\ 0}$. Then $L_0=\{tI_0:t\in\qp\}$ is
the space of $\{w_1(t)\}$-fixed vectors.

\begin{lema} \label{lema:qp2}
  Let $\{p_i\}\subset\qp^2\setminus L_0$ be a sequence such that
  $p_i\to I_0$ as $i\to\infty$. Then, after passing to a subsequence,
  there exists a sequence $t_i\to\infty$ such that the following
  holds: Then for any $s\in\qp$,
  \begin{equation*}
    \lim_{i\to\infty} w_1(st_i)\cdot p_i= (1+s)I_0.
  \end{equation*}
\end{lema}

\begin{proof} 
  Write $p_i=\smat{a_i \\ b_i}$, $\forall i$. Since $p_i\not\in L_0$,
  $b_i\neq 0$, $\forall i$. Put $t_i=b_i\inv$. Then for any $s\in\qp$,
  as $i\to\infty$,
  \begin{equation*}
    w_1(st_i)\smat{a_i \\ b_i}=\smat{a_i+s \\ b_i}\to
    \smat{1+s \\ 0}.
  \end{equation*}
\end{proof}

Let $I_1=\smat{1 & \\ & 1}$. For $1\leq m\leq n$,
put
\begin{equation} \label{eq:E_m}
  \begin{array}{ll}
    E_m&=M_2(\qp)^m \\ 
    I_m&=(I_1,\ldots,I_1)\in E_m \\
    w_m(t)&=(w_1(t),\ldots,w_1(t))\in \SL_2(\qp)^m \\
    L_m&=\{X\in E_m: w_m(t) X w_m(-t)=X,\, \forall t\in\qp\}=(L_1)^m,
  \end{array}
\end{equation}

\begin{lema}[Margulis] \label{lema:conj:u}
  Let $\{X_i\}\subset E_m\setminus L_m$ be a sequence such that
  $X_i\to I_m$ as $i\to\infty$. Then after passing to a subsequence,
  there exist a sequence $t_i\to\infty$, and a nonconstant polynomial
  map $\psi:\qp\to\qp^m$ of degree at most $2$ such that given any
  $s\in\qp$ and a sequence $s_i\to s$ in $\qp$,
  \begin{equation} \label{eq:cor:conj}
    \lim_{i\to\infty} w_m(s_it_i)X_i
    w_m(-s_it_i)=w_m(\psi(s)).
  \end{equation}
  In particular, $\psi(0)=0$.
\end{lema}

\begin{proof} 
  If we write $X_i=(X_i(1),\ldots,X_i(m))$, where $X_i(j)\in M_2(\qp)$
  for $1\leq j\leq m$, and $X_i(j,t)=w_1(t)X_i(j)w_1(-t)$ then
  \begin{equation*}
    w_m(t)X_iw_m(-t)=(X_i(1,t),\ldots,X_i(m,t)).
  \end{equation*}
  Fix any $1\leq j\le m$. If
  $X_i(j)=\smat{a_i(j)&b_i(j)\\c_i(j)&d_i(j)}$, then
  \begin{equation} \label{eq:Xijt}
    X_i(j,t)= X_i(j) + 
    \begin{pmatrix}
      c_i(j) & d_i(j)-a_i(j) \\
      0      & -c_i(j)
    \end{pmatrix} t + 
    \begin{pmatrix} 
      0 & -c_i(j) \\
      0 & 0
    \end{pmatrix} t^2,
  \end{equation}
  If $X_i(j,t)=X_i(j)$ for all $t$, then $c_i(j)=0=d_i(j)-a_i(j)$, and
  we put $t_i(j)=\infty$. If $c_i(j)\neq 0$ or $d_i(j)-a_i(j)\neq 0$,
  then there exists $t_i(j)\in\qp$ such that
  \begin{equation} \label{eq:tij}
    \max\bigl\{\abs{(d_i(j)-a_i(j))t_i(j)},\abs{c_i(j)t_i(j)^2}\bigr\}=1.
  \end{equation}
  As $i\to\infty$, since $X_i(j)\to 0$, we have $a_i(j)-d_i(j)\to
  1-1=0$ and $c_i(j)\to 0$. Therefore $t_i(j)\to\infty$, and hence
  $\abs{c_i(j)t_i(j)}\leq \abs{t_i(j)}\inv\to 0$ as $i\to\infty$.
  
  Put
  \begin{equation} \label{eq:t_i}
    t_i=\min\{t_i(1),\ldots,t_i(m)\}.
  \end{equation}
  Since $X_i\not\in (L_1)^m$, we have that $t_i<\infty$. Since $X_i\to
  I_m$, we have $t_i\to\infty$. By \eqref{eq:tij} and \eqref{eq:t_i},
  after passing to a subsequence, for each $1\leq j\leq m$, there exist
  $\alpha_j,\beta_j\in\qp$ such that
  \begin{equation} \label{eq:ab}
    \lim_{i\to\infty} (d_i(j)-a_i(j))t_i=\alpha_j \quad \text{and} \quad 
    \lim_{i\to\infty} -c_i(j)t_i^2 = \beta_j.
  \end{equation}
  In particular, $c_i(j)t_i\to 0$ for all $j$. Now \eqref{eq:cor:conj}
  follows from \eqref{eq:Xijt} and \eqref{eq:ab}, where
  \begin{equation*}
    \psi(s)=(\alpha_1 s+\beta_1 s^2,\ldots,\alpha_m s + \beta_m s^2).
  \end{equation*}
  Due to \eqref{eq:tij}, $\abs{\alpha_{j_0}}=1$ or $\abs{\beta_{j_0}}=1$
  for some $j_0$. Therefore $\psi$ is nonconstant. 
\end{proof}

\subsection*{Proof of Proposition~\ref{prop:UMU}:}

Let
\begin{equation*}
  E=E_{n-1}\times \qp^2 \quad \text{and} \quad \vp=(I_{n-1};I_0).
\end{equation*}
Define the linear action of $G$ on $E$ as follows: For any\\
$g=(g(1),\ldots,g(n))\in G$, and $X=(X(1),\ldots,X(n-1);Y)\in E$,
\begin{equation} \label{eq:gX}
  g\cdot X=(g(1)X(1)g(2)\inv, \ldots, g(n-1)X(n-1)g(n)\inv;g(n)Y).
\end{equation}
Then
\begin{equation} \label{eq:stab:U}
  U=\{g\in G: g\cdot \vp=\vp\}.
\end{equation}
Let $L=\{X\in E: U\cdot X=X\}=L_{n-1}\times L_0$. Then
\begin{equation} \label{eq:nor}
  \nor(U)=\{g\in G: g\cdot p\in L\}.
\end{equation}
We note that $\nor(U)=Z(G)DW$, where $Z(G)=\{(\pm I_1,\dots,\pm
I_1)\in G\}$ is the center of $G$. Also
\begin{equation} \label{eq:open-orbit}
  G\cdot\vp=\SL_2(\qp)^{n-1}\times \qp^\ast I_0.
\end{equation} 

For $g\in G$ if $g\cdot\vp\in \cl{UM\cdot\vp}$ then there exist
$u_i\in U$ and $m_i\in M$ such that $u_im_i\cdot\vp\to g\cdot\vp$ as
$i\to\infty$. Then $(g\inv u_i m_i)\cdot\vp\to p$. Therefore, by
\eqref{eq:open-orbit} there exists a sequence $\delta_i\to e$ in $G$
such that $(g\inv u_im_i)\cdot\vp=\delta_i\cdot\vp$ for all $i$. By
\eqref{eq:stab:U} there exist $u_i'\in U$ such that $g\inv
u_im_iu_i'=\delta_i$ for each $i$. Therefore $u_im_iu_i'\to g$.

Thus for any $g\in G$,
\begin{equation} \label{eq:gp:DW}
  g\cdot\vp\in \cl{UM\cdot\vp}\cap L
  \Leftrightarrow 
  g\in\cl{UMU}\cap\nor(U).
\end{equation}

By \eqref{eq:nor}, $M\cdot\vp\cap L=\emptyset$ and $e\in\cl{M}$.
Therefore there exists a sequence
\begin{equation} \label{eq:Xi:Mp}
  \{X_i\}\subset M\cdot\vp\subset E\setminus L,
\end{equation} 
such that $X_i\to \vp$ as $i\to\infty$. By combining
Lemma~\ref{lema:qp2} and Lemma~\ref{lema:conj:u}, after passing to
subsequences, there exists a sequence $t_i\to\infty$ in $\qp$ such
that for any $s\in\qp$,
\begin{equation} \label{eq:conj:u:quote}
  \lim_{i\to\infty} u(st_i)\cdot X_i =
  (w_{n-1}(\psi(s));\phi(s)I_0)\in L,
\end{equation}
were $\phi(s)$ is a polynomial of degree at most $1$, $\phi(0)=1$ and
\begin{equation*}
  \psi(s)=(\psi_1(s),\ldots,\psi_{n-1}(s))\in\qp^{n-1}
\end{equation*}
is a polynomial map of degree at most $2$, $\psi(0)=\boldsymbol{0}$,
and $\psi$ or $\phi$ is non-constant.  We define
$\psi'_k=\sum_{j=k}^{n-1}\psi_j$ for $1\leq k\leq n-1$, and
\begin{equation*}
  \psi'(s)=(\psi'_1(s),\ldots,\psi'_{n-1},0)\in\qp^n.
\end{equation*}
Then $\psi':\qp\to\qp^n$ is a polynomial of degree at most $2$, and
$\psi'$ is constant if and only if $\psi$ is constant.

For any $s\in\qp$ such that $\phi(s)\neq 0$, we put
\begin{equation} \label{eq:Phi:DUperp}
  \Phi(s)=w(\psi'(s))d(\phi(s)).
\end{equation}
Therefore due to \eqref{eq:gX},
\begin{equation} \label{eq:WDp}
  \Phi(s)\cdot\vp=(w_1(\psi_1(s)),\ldots,w_1(\psi_{n-1}(s));\phi(s)I_0)
  \in L.
\end{equation}
Therefore by \eqref{eq:Xi:Mp}--\eqref{eq:WDp},
\begin{equation*}
  \Phi(s)\cdot\vp\in \cl{U\cdot(M\cdot\vp)}\cap L.
\end{equation*}
Hence by \eqref{eq:gp:DW} and \eqref{eq:Phi:DUperp}, for all $s\in\qp$
with $\phi(s)\neq 0$,
\begin{equation*}
  \Phi(s)\in DU^\perp\cap \cl{UMU}.
\end{equation*}
Now the conclusion of the proposition follows from
Lemma~\ref{lema:DU-perp} proved below.  \qed

\subsection{Some more elementary lemmas}

It is straightforward to verify the following.

\begin{lema} \label{st-q}
  Let $m\in\N$ and $\psi:\qp\to \qp^m$ be a polynomial map such that
  $\deg(\psi)\geq 1$.  Then there exists a nonzero vector
  $\vv\in\qp^m$ such for any $s\in \qp$,
  \begin{equation} \label{eq:st-q}
    \psi(t+st^{-q})-\psi(t)\to s\vv \qquad \mbox{as $t\to\infty$},
  \end{equation}
  where $q=\deg(\psi)-1$. In particular, any closed additive subgroup
  generated by $\psi(\qp)$ contains a nonzero subspace of $\qp^m$.
  \qed
\end{lema} 

\begin{lema} \label{lema:abelian} 
  Let $F$ be an abelian subgroup of $DW$ the either $F\subset \{d(\pm
  1)\}W$ or there exists $v\in W$ such that $F\subset vDv\inv$.
\end{lema}

\begin{proof}
  Suppose $d(\alpha)w(\vt)\in F$ for some $\alpha\in \qp^\ast$ such
  that $\alpha=\neq \pm 1$. Let $v=w((1-\alpha^2)\inv \vt)$. Then
  $v\inv d(\alpha)w(\vt) v=d(\alpha)$. Therefore $v\inv Fv$ is
  contained in the centralizer of $d(\alpha)$.

  Now for any $\beta\in\qp^\times$ and $\vs\in\qp^n$, we have
  \begin{equation*}
    d(\alpha)[d(\beta)w(\vs)]d(\alpha)\inv=d(\beta)w(\alpha^2\vs).
  \end{equation*}
  Therefore, since $\alpha^2\neq 1$, we have $vFv\inv\subset D$.
\end{proof}

\begin{lema}
  \label{lema:DU-perp}
  Let $\phi:\qp\to\qp$ be a linear map, and
  $\psi:\qp\to\qp^{n-1}\times\{0\}$ be a polynomial map such that at
  least one of them is non-constant, $\phi(0)=1$ and $\psi(0)=0$.  Let
  $F$ be the closed subgroup of $DU^\perp$ generated by
  \begin{equation*}
    \{\Phi(t):=w(\psi(t))d(\phi(t)):t\in\qp,\,\phi(t)\neq 0\}.
  \end{equation*}
  Then either $F$ contains a nontrivial one-parameter subgroup of
  $U^\perp$ or $F=vDv\inv$ for some $v\in U^\perp$.
\end{lema}

\begin{proof} 
  If $F\subset U^\perp$ then the result follows from Lemma~\ref{st-q}.
  Otherwise $\phi$ is a non-constant linear map.  Therefore
  $\phi(\qp)=\qp$. In particular, $F\not\subset Z(G)W$.
  
  If $F$ is abelian, then by Lemma~\ref{lema:abelian} there exists
  $v\in W$ such that $F\subset vDv\inv$. Since $\phi$ is linear and
  nonconstant, $F=vDv\inv$.
  
  Now we can further assume that $F$ is not abelian.
  Since the commutator
  \begin{equation*}
    [F,F]\subset [DU^\perp,DU^\perp]\subset U^\perp,
  \end{equation*}
  there exists $\vs\in\qp^{n-1}\times\{0\}$, $\vs\neq 0$ such that
  $w(\vs)\in F$. Therefore
  \begin{equation*}
    \Phi(t)w(\vs)\Phi(t)\inv =w(\phi(t)^2\vs)\in F, \qquad\forall
    t\in\qp.
  \end{equation*}
  Put $\tilde{\psi}(t):=\phi^2(t)\vs$. Then $\tilde\psi:\qp\to\qp^{n-1}\times
  \{0\}$ is a non-constant polynomial map.  Therefore by
  Lemma~\ref{st-q} applied to $\tilde\psi$ we conclude that $F$
  contains a nontrivial one-parameter subgroup of $U^\perp$. This
  completes the proof.  
\end{proof}

The following is a special case of the general fact that cocompact
discrete subgroups in semisimple Lie groups do not contain unipotent
elements having nontrivial Adjoint action on the Lie algebra.

\begin{prop} \label{prop:cocomp}
  $W\cap G_x=\{e\}$ for all $x\in G/\Gamma$.
\end{prop}

\begin{proof} 
  Let $C$ be a compact subset of $G$ such that $C\Gamma=G$. Since
  $\Gamma$ is discrete, there exists a neighbourhood $\Omega$ of $e$
  in $G$ such that $cZ(G)\Gamma\inv \cap \Omega=\{e\}$ for all $c\in
  C$.  Therefore
  \begin{equation*}
    G_y\cap \Omega=\{e\}, \qquad \forall\, 
    y\in C\Gamma/\Gamma=G/\Gamma.
  \end{equation*}
  Suppose that $w(\vt)\in G_x$ for some $\vt\in\qp$.
  Let $\alpha\in\qp^\times$ such that $\abs{\alpha}<1$. Then
  \begin{equation*}
    G_{d(\alpha^i)x}=d(\alpha^i)G_x d(\alpha^{-i})\ni 
    d(\alpha^i)w(\vt)d(\alpha^{-i})=w(\alpha^{2i}\vt)\to e
  \end{equation*}
  as $i\to\infty$. Therefore $w(\alpha^{2i}\vt)\in 
  G_{d(\alpha^i)x}\cap \Omega =\{e\}$ for some $i$. Hence $w(\vt)=e$.
\end{proof}

\begin{prop} \label{prop:DW/Delta} Let $\Delta$ be a discrete subgroup
  of $DW$ such that $\Delta\cap W=\{e\}$. Then $W$ acts properly on
  $DW/\Delta$.
\end{prop}

\begin{proof} 
  We have
  \begin{equation*}
    [\Delta,\Delta]\subset [DW,DW]\cap \Delta\subset W\cap \Delta=\{e\}.
  \end{equation*}
  Hence $\Delta$ is an abelian subgroup of $DW$. If $g=d(-1) w(\vt)\in
  \Delta$ for some $\vt\in \qp^n$, then $g^2=w(2\vt)\in \Delta\cap
  W=\{e\}$; and hence $\vt=0$. Therefore by Lemma~\ref{lema:abelian}
  there exists $v\in W$ such that $\Delta \subset vDv\inv$.
  
  Since $DW=(vDv\inv)W=W(vDv\inv)$, we have that
  \begin{equation*}
  DW/\Delta=W(vDv\inv)/\Delta\cong W\times (vDv\inv/\Delta)
  \end{equation*}
  is a $W$-equivariant isomorphism, where $W$ acts on the space
  $W\times (vDv\inv/\Delta)$ by translation on the first factor and
  trivially on the second factor; and this action is proper.
\end{proof}

\section{$U$-minimal sets}

In order to understand closed $U$-invariant sets, especially the
closures of $U$-orbits, we begin with the study of $U$-minimal sets.

\begin{theo} \label{thm:U-minimal}
  Let $X$ be a $U$-minimal subset of $G/\Gamma$.  Then $X$ is
  invariant under either $vDv\inv$ for some $v\in U^\perp$ or a
  nontrivial one-parameter subgroup of $U^\perp$.
\end{theo}

\begin{proof} 
  Let $M=\{g\in G: gX\cap X\neq\emptyset\}$. For any $g\in
  M\cap\nor(U)$, $gX\cap X$ is a nonempty closed $U$-invariant set.
  Hence by minimality $gX=X$. Thus $M\cap \nor(U)$ is a closed
  subgroup of $G$. We note that $DW$ is a subgroup of finite index in
  $\nor(U)=Z(G)DW$.  Therefore $M_1:=M\cap DW$ is a closed subgroup of
  $DW$ and an open subgroup of $M\cap\nor(U)$.

  First suppose that $e\not\in \cl{M\setminus \nor(U)}$. Then every
  orbit of $M\cap \nor(U)$ in $X$ is open. Therefore every orbit of
  $M_1$ on $X$ is open, and hence it is compact. Let $x\in X$. Since
  $U\subset M_1$, and $X$ is $U$-minimal, $X=M_1x$. Hence
  $M_1/(M_1)_x\cong M_1x=X$ is compact. By
  Proposition~\ref{prop:cocomp} and Proposition~\ref{prop:DW/Delta},
  $U$ acts properly on $M_1/(M_1)_x$, which is a contradiction.

  Therefore $e\in\cl{M\setminus\nor(U)}$. By Lemma~\ref{lema:top1},
  $X$ is invariant under the subgroup generated by
  $\nor(U)\cap\cl{UMU}$.  Now the conclusion of the theorem follows
  from Proposition~\ref{prop:UMU}.
\end{proof}

\begin{cor} \label{cor:n1} 
  Let $n=1$; that is, $G\cong\SL_2(\qp)$ and $\Gamma$ is a cocompact
  discrete subgroup of $G$. Then $Ux$ is dense in $G/\Gamma$ for every
  $x\in G/\Gamma$.
  
  In other words, Theorem~\ref{thm:closure} is valid for $n=1$. 
\end{cor}

\begin{proof} 
  Since $\cl{Ux}$ is a closed $U$-invariant subset of $G/\Gamma$,
  there exists a compact $U$-minimal subset $X\subset\cl{Ux}$. By
  Theorem~\ref{thm:U-minimal}, $X$ is invariant under $D$, because for
  the case of $n=1$, we have $W=U$ and $U^\perp=\{e\}$. Thus $X$ is a
  closed $DW$-invariant subset of $G/\Gamma$. Therefore by
  Proposition~\ref{prop:horo}, $X=G/\Gamma$. Thus $\cl{Ux}=G/\Gamma$.
\end{proof}

\subsection{$D$-invariant $U$-minimal sets}
In view of Theorem~\ref{thm:U-minimal}, we first suppose that the
$U$-minimal set is invariant under $wDw\inv$ for some $w\in U^\perp$.
Now $Y:=w\inv X$ is $U$-minimal and $D$-invariant.  Therefore for
simplicity of notation we will further investigate $Y$, rather than
$X$.

We need the following group theoretic result.

\begin{prop} \label{prop:UhB}
  Let sequences $\{h_i\}$ in $\SL_2(\qp)$ and $\{t_i\}$ in $\qp$,
  $\abs{t_i}\to\infty$ be given. Then, after passing to a subsequence,
  there exists at most one $s^\ast\in\qp$ such that for any $s\in\qp$,
  $s\neq s^\ast$, the following holds:
  \begin{equation} \label{eq:UhB}
    w_1(st_i)h_i B\to eB \qquad \text{as $i\to\infty$}, 
  \end{equation}
  where $B$ is the group of all upper triangular matrices in
  $\SL_2(\qp)$, and the limit is considered in the quotient space
  $\SL_2(\qp)/B$.
  
  In fact, if $\{h_i\}$ is a constant sequence, then \eqref{eq:UhB}
  holds for all $s\in\qp$.
\end{prop}

\begin{proof} 
  Consider the projective linear action of $\SL_2(\qp)$ on the
  projective space $\cP=(\qp^2\setminus\{0\})/\qp^\times$. Let
  $\pline{\vv}$ denote the image of $\vv\in\qp^2$ on $\cP$. The the
  stabilizer of $\pline{{\ve}_1}$ is $B$, where $\ve_1=\smat{1 \\ 0}$.  We
  can express $h_i\pline{e_1}= \pLine{\smat{a_i\\b_i}}$, where
  $\abs{a_i}^2+\abs{b_i}^2=1$. Then for any $s\in\qp$,
  \begin{equation} \label{eq:proj}
    w_1(st_i)h_i\pline{e_1}=
    \pLine{\smat{a_i+st_ib_i \\ b_i}}
    = \pLine{\smat{1 \\ b_i/a_i+st_ib_i}}, \qquad 
    \text{if $s\neq -a_i/(t_ib_i)$}.
  \end{equation}

  After passing to a subsequence, either $-a_i/(t_ib_i)\to s^\ast$ for
  some $s^\ast\in\qp$, or $\abs{-a_i/(t_ib_i)}\to\infty$. By
  \eqref{eq:proj}, if $s\neq s^\ast$, then since $\abs{t_i}\to\infty$,
  \begin{equation*}
    w_1(st_i)h_i\pline{e_1}\to\, \pline{e_1}.
  \end{equation*}
  From this \eqref{eq:UhB} follows, because the action of $\SL_2(\qp)$
  on $\cP$ is transitive, and the stabilizer of $\pline{e_1}$ is $B$.
\end{proof}

The next proposition is very similar to Proposition~\ref{prop:UMU},
and it will allow us to investigate further after an application of
Lemma~\ref{lema:top1}.

\begin{prop} \label{prop:DUMDU}
  Let $M\subset G$ such that $e\in\cl{M\setminus H}$. Then the closed
  subgroup generated by $\cl{DUMDU}\cap W$ contains a nontrivial
  one-parameter subgroup of $U^\perp$.
\end{prop}

\begin{proof} 
  First we suppose that $e\not\in\cl{M\setminus U^\perp H}$.  Since
  $e\in\cl{M\setminus H}$, there exit $v\in U^\perp\setminus\{e\}$ and
  $h\in H$ such that $v h\in M$. By Proposition~\ref{prop:UhB},
  applied to $H\cong \SL_2(\qp)$ and $DU\cong B$, there exists a
  sequence $\{u_i\}\subset U$ such that $u_ihDU\to eDU$ in $H/DU$.
  Hence
  \begin{equation*}
    u_ivhDU=vu_ihDU\to vDU, \quad \text{as $i\to\infty$}.
  \end{equation*}
  Therefore $v\in \cl{UMDU}$. We can write $v=w(\vt)$,
  $\vt\in\qp^n\setminus\{0\}$. Then $d(a) v d(-a)=w(a^2\vt)$ for all
  $a\in\qp^\times$. By Lemma~\ref{st-q}, the closure of the additive
  subgroup generated by $\{a^2\vt:a\in\qp\}$ in $\qp^n$ contains
  $\qp\vt$. Hence the subgroup generated by $\cl{DUMDU}\cap W$
  contains a nontrivial one-parameter subgroup of $U^\perp$.

  Now we may assume that $e\in\cl{M\setminus U^\perp H}$. Let a
  sequence $\{g_i\}\subset M\setminus U^\perp H$ be such that $g_i\to
  e$. Since $G=G_{\{1,\ldots,n-1\}}H$, we can write $g_i=X_ih_i$,
  where $X_i\in G_{\{1,\ldots,n-1\}}\setminus U^\perp$, $X_i\to 0$,
  and $h_i\to e$ in $H$. By Lemma~\ref{lema:conj:u}, after passing to
  a subsequence, there exist a sequence $t_i\to\infty$ in $\qp$ and a
  non-constant polynomial map $\psi:\qp\to\qp^n$ of degree at most $2$
  such that for any $s\in\qp$,
  \begin{equation} \label{eq:uXu}
    \lim_{i\to\infty} u(st_i)X_iu(-st_i) =
    w(\psi(s))\in U^\perp.
  \end{equation}
  
  By Proposition~\ref{prop:UhB}, there exists at most one
  $s^\ast\in\qp$ such that for all $s\in \qp$ with $s\neq s^\ast$, the
  following holds:
  \begin{equation} \label{eq:uhDU}
    u(st_i)h_i(DU)\to DU, \qquad \mbox{as $i\to\infty$.}
  \end{equation}
  
  By \eqref{eq:uXu} and \eqref{eq:uhDU}, $\forall s\in\qp$ with $s\neq
  s^\ast$, as $i\to\infty$,
  \begin{equation*}
    u(st_i)g_iDU=(u(st_i)X_iu(-st_i))(u(st_i)h_iDU)\to w(\psi(s))DU,
  \end{equation*}
  in $G/DU$.  Thus $w(\psi(s))\in \cl{UMDU}$, $\forall s\in\qp$. Since
  $W\cong \qp^n$, and $\psi(s)$ is a non-constant polynomial map, the
  conclusion of this proposition follows from Lemma~\ref{st-q}.
\end{proof}

\begin{theo} \label{thm:DU-minimal}
  Let $X$ be a $U$-minimal subset of $G/\Gamma$. Then either $X$ is a
  closed orbit of $wHw\inv$ for some $w\in U^\perp$, or $X$ is
  invariant under a nontrivial one-parameter subgroup of $U^\perp$.
\end{theo}

\begin{proof} 
  By Theorem~\ref{thm:U-minimal}, we are reduced to considering the
  case that $X$ is $wDw\inv$-invariant for some $w\in W$.

  We put $Y=w\inv X$. Then $Y$ is $DU$-invariant and $U$-minimal.  Let
  \begin{equation*}
    M=\{g\in G: gY\cap Y\neq\emptyset\}.
  \end{equation*} 
  
  By Lemma~\ref{lema:top1}, applied to $Y'=Y$, $P=P'=DU$ and $F=U$, we
  have that $Y$ is invariant under the subgroup generated by
  $\cl{DUMDU}\cap \nor(U)$.
  
  Now if $e\in\cl{M\setminus H}$ then by Proposition~\ref{prop:DUMDU},
  there exists a nontrivial one-parameter subgroup, say $V$, of
  $U^\perp$ such that $VY=Y$. Therefore
  \begin{equation*}
    VX=V(w\inv Y)=w\inv(VY)=w\inv Y=X,
  \end{equation*}
  and the conclusion of the theorem holds.
  
  Next suppose that $e\not\in \cl{M\setminus H}$. Fix $y\in Y$ and let
  $\Delta=H_y$. Then by Lemma~\ref{lema:top2}, $\cl{DU\Delta}/\Delta$ is
  compact in $H/\Delta$. Since $H/DU$ is compact, we have that
  $H/\Delta$ is compact. Therefore by Proposition~\ref{prop:horo}
  applied to the case of $G:=H\cong \SL_2(\qp)$, $\Gamma:=\Delta$,
$W:=U$, and $D:=D$, we conclude that $\cl{DU\Delta}=H$. Since $Hy\cong
H/\Delta$, we have that $Hy$ is compact and $Hy=\cl{DUy}=Y$. Hence
$X=(w Hw\inv)(wy)$, which is a closed orbit of $wHw\inv$.
\end{proof}

\subsection{Minimal sets for actions of at least $2$ dimensional
  subgroups of $W$}

\begin{rem} \label{rem:prod} 
  For any $x\in G/\Gamma$, there exists
  $g_j\in G_{\{\{j\}\}}$ for $1\leq j\leq n$ such that $x=(g_1\dots
  g_n)\Gamma$, and
  \begin{equation*}
    G_x=(g_1\Gamma_1g_1\inv)\cdots(g_n\Gamma_n g_n\inv).
  \end{equation*}
  In particular, for any $J\subset\{1,\ldots,n\}$, we have
  \begin{equation} \label{eq:GJ_x}
    (G_J)_x\cong \prod_{j\in J} G_{\{j\}}/g_j\Gamma_jg_j\inv.
  \end{equation}
\end{rem}

For $\cJ\in\cC$, define $\cup\cJ=\cup_{J\in\cJ} J$, $\num{\cJ}=\num{\cup\cJ}$,
$G_{\cJ}=G_{\cup\cJ}$, $W_{\cJ}=W\cap G_{\cJ}$, $U_{\cJ}=W\cap
H_{\cJ}=\prod_{J\in\cJ} U_J$, where $U_J=W\cap H_J$, and
$D_{\cJ}=A\cap H_{\cJ}=\prod_{J\in\cJ} D_J$, where $D_J=D\cap H_J$.

\begin{theo} \label{thm:std} 
  Assume that for any $k<n$ the Theorem~\ref{thm:closure} is true for
  $k$ in place of $n$. Let $\cJ\in \cC$ such that
  $\cJ\neq\{\{1,\ldots,n\}\}$. Then for any $x=G/\Gamma$, we have
  $\cl{U_{\cJ}x}=wH_{\cJ'}w\inv x$ for some $\cJ'\in\cC$ and $w\in W$.
\end{theo}

\begin{proof} 
  We intend to prove this result by induction on $n$.
  
  By our choice of $\cJ$ there exists $\cJ_1\subset \cJ$, where
  $\cJ_1\in \cC$ and $1\leq n_1:=\num{\cup\cJ_1}<n$. Put $G_1=G_{\cJ_1}$
  and $U_1=U_{\cJ_1}$.
  
  By Remark~\ref{rem:prod}, for any $y\in G/\Gamma$,
  \begin{equation*}
    G_1 y\cong \prod_{j\in\cup\cJ_1} G_{\{j\}}/(G_{\{j\}})_y.
  \end{equation*} 
  
  We claim that there exists $\cJ_1'\in\cC$ and $w_1\in W_{\cJ_1}$ such
  that, if we put $H_1=w_1H_{\cJ_1'}w_1\inv$ then
  \begin{equation} \label{eq:J1}
    \cl{U_1x}=H_1 x.
  \end{equation}
  Here $U_1\subset H_1\subset G_1$.
  
  If $\cJ_1=\{\{\cup \cJ_1\}\}$, then the claim follows by applying the
  assumption that Theorem~\ref{thm:closure} is valid for $n_1<n$, $G_1$
  in place of $G$, and $U_1$ in place of $U$.
  
  If $\cJ_1\neq \{\{\cup \cJ_1\}\}$ then the claim follows by applying
  the induction hypothesis of this theorem to $n_1<n$ in place of $n$,
  $G_1$ in place of $G$, and $\cJ_1$ in place of $\cJ$. Thus the claim
  is proved in all the cases.
  
  If $\cJ\not\in\cC_0$, then $\num{\cJ}<n$, and hence if we choose
  $\cJ_1=\cJ$ then the conclusion of the theorem follows from
  \eqref{eq:J1}.
  
  Therefore we can assume that $\cJ\in\cC_0$. Let $\cJ_2=\cJ\setminus
  \cJ_1\neq\emptyset$. Then
  $n_2=\num{\cJ_2}=\num{\cJ}-\num{\cJ_1}=n-n_1<n$. By the same argument
  as above for $J_2$ in place of $J_1$ the following holds: there exists
  $\cJ_2'\in\cC$ and $w_2\in W_{\cJ_2'}$ such that, if we put
  $G_2=G_{\cJ_2}$, $U_2=U_{\cJ_1}$ and $H_2=w_2H_{\cJ_2'}w_2\inv$, then
  $U_2\subset H_2\subset G_2$ and
  \begin{equation} \label{eq:J2}
    \cl{U_{\cJ_2}x}=H_2 x.
  \end{equation}
  
  Since $(\cup\cJ_1)\cap(\cup\cJ_2)=\emptyset$, we have that
  $G_{\cJ_1}\subset Z_G(G_{\cJ_2})$. Therefore for any $g_2\in
  G_{\cJ_2}$, we have
  \begin{equation} \label{eq:g2}
    \cl{U_{\cJ_1}g_2x}=\cl{g_2 U_{\cJ_1} x}=g_2(w_1H_{\cJ_1'}w_1\inv x)
    =w_1H_{\cJ_1'}w_1\inv (g_2x).
  \end{equation}
  Moreover
  \begin{equation} \label{eq:H1H2}
    H_1H_2x\cong H_1/(H_1)_x\times H_2/(H_2)_x
  \end{equation}
  which is compact. Hence by \eqref{eq:J1}--\eqref{eq:H1H2},
  \begin{equation*}
    \cl{U_{\cJ}x}=\cl{U_1U_2x}=\cl{U_1H_2 x}=H_1H_2x=wH_{\cJ'}w\inv x,
  \end{equation*}
  where $w=w_1w_2$ and $\cJ'=\cJ_1'\cup \cJ_2'$. This completes the
  proof of the theorem.  
\end{proof}  

\begin{rem}
  By the condition of Theorem~\ref{thm:std}, $n\geq 2$. Therefore to
  begin the induction, we have $n=2$ and for this case
  $\cJ=\{\{1\},\{2\}\}$, $\cJ_1=\{\{1\}\}$ and $\cJ_2=\{\{2\}\}$, and
  the result follows from the assumption that
  Theorem~\ref{thm:closure} is valid for $n=1$; in fact, this
  assumption was verified in Corollary~\ref{cor:n1}.
\end{rem}

\begin{theo}  \label{thm:multi}
  Assume that for all $k<n$, the Theorem~\ref{thm:closure} is true for
  $k$ in place of $n$. Let $V$ be a multi-parameter subgroup of $W$ of
  dimension at least $2$ and containing $U$. Let $X$ be a compact
  $V$-minimal subset of $G/\Gamma$. Then there exists $\cJ\in\cC_0$
  and $w\in W$ such that $X=(wH_{\cJ}w\inv)x$.
\end{theo}

\begin{proof} 
  Without loss of generality we may assume that $V$ is the
  largest multi-parameter subgroup of $W$ whose action preserves $X$.
  
  If $n=2$ then $V=W$ and the theorem follows from
  Theorem~\ref{thm:std}. We intend to prove this theorem by induction on
  $n$.
  
  Let $V_1=V\cap G_{\{1,\ldots,n-1\}}$. Then $V=V_1U$, and $\dim V_1\geq
  1$ (see Definition~\ref{def:multi-para}). Let $J$ be the smallest
  subset of $\{1,\ldots,n-1\}$ such that $V_1\subset G_J$. Then there
  exists $a\in A\cap G_J$ such that $U_J \subset aV_1a\inv$.
  
  Let $Y$ be a compact $V_1$-minimal subset of $X$. Take $y\in Y$. Then
  by Remark~\ref{rem:prod}, $G_Jy$ is compact, and 
  \begin{equation*}
    G_Jy\cong G_J/\prod_{j\in J} (G_{\{j\}})_y.  
  \end{equation*}
  In particular, $Y\subset G_Jy$.
  
  Let $y_1=ay$. We claim that there exists $\cJ\in\cC$ and $w_1\in W\cap
  G_J$ such that $H_\cJ\subset G_J$ and
  \begin{equation} \label{eq:multi2}
    \cl{(a V_1 a\inv) y_1}=w_1H_{\cJ}w_1\inv y_1.
  \end{equation}
  
  If $\dim V_1=1$, then $U_J=aV_1a\inv$. Since $\num{J}<n$, the claim
  follows from our first hypothesis that Theorem~\ref{thm:closure} is
  valid for $\num{J}$ in place of $n$, $G_J$ in place of $G$, and $U_J$ in
  place of $U$.
  
  If $\dim V_1\geq 2$, then the claim follows from the induction
  hypothesis of this theorem applied to $G_J$ in place of $G$ and
  $aV_1a\inv$ in place of $V$ in the statement. This completes the proof
  of the claim in both the cases.
  
  From \eqref{eq:multi2} we have that
  \begin{equation*}
    Y\supset \cl{V_1y}=\cl{V_1 a\inv y_1} = a\inv
    \cl{(aV_1a\inv)y_1}=a\inv w_1H_{\cJ}w_1\inv a y.
  \end{equation*}
  Let $w=a\inv w_1 a\in W\cap G_J$. Then $w D_{\cJ}w\inv y\subset
  Y\subset X$, where $D_{\cJ}=A\cap H_\cJ=\prod_{I\in\cJ}D_I$.
  
  Therefore by Lemma~\ref{lema:top1},
  \begin{equation} \label{eq:mult1}
    gX=X, \quad \forall g\in N_G(V)\cap \cl{V(wD_\cJ w\inv)V}.
  \end{equation}
  
  We have
  \begin{equation} \label{eq:mult2}
    N_G(V)\cap\cl{V(wD_\cJ w\inv)V}\supset W\cap \cl{UD_{\cJ}U},
  \end{equation}
  and
  \begin{equation} \label{eq:mult4}
    \cl{UD_{\cJ}U}\supset \cl{\{uD_{\cJ}u\inv:u\in
      U\}}=\prod_{I\in\cJ} \cl{\{uD_{I}u\inv:u\in U_{I}\}}.
  \end{equation}
  Take any $I\in\cJ$. Let $\{X_i\}$ be a sequence in $D_I\setminus\{e\}$
  such that $X_i\neq e$ as $i\to \infty$. In view of the identification
  $D_I\subset H_I\cong\SL_2(\qp)\subset \M_2(\qp)=E_1$, we have that
  \begin{equation*}
    \{X_i\}\subset E_1\setminus L_1
  \end{equation*}
  (recall \eqref{eq:E_m}).  We apply Lemma~\ref{lema:conj:u} to conclude
  the following: The subgroup generated by $\cl{\{uD_{I}u\inv:u\in
    U_{I}\}}$ contains $U_I$ (see Lemma~\ref{st-q}). Therefore by
  \eqref{eq:mult4}, the subgroup generated by $\cl{UD_{\cJ}U}$ contains
  $U_{\cJ}$. Therefore by \eqref{eq:mult1} and \eqref{eq:mult2},
  $U_{\cJ}X=X$. By the maximality of $V$, assumed in the beginning of
  the proof, $U_{\cJ}\subset V$. Thus $U_\cJ\subset G_J\cap V=V_1$.
  Therefore
  \begin{equation*}
    U_\cJ\subset V_1 \subset a\inv H_{\cJ}a\cap W=a\inv(H_\cJ \cap W) a=a
    U_\cJ a\inv.
  \end{equation*}
  Therefore $U_\cJ=a\inv U_\cJ a$, and hence $V_1=U_\cJ$. Thus $V=U_\cJ
  U=U_{\cJ'}$, where $\cJ'=\cJ\cup\{\{1,\ldots,n\}\}$. Now the theorem
  follows from Theorem~\ref{thm:std}.
\end{proof}

\section{Proof of Theorem~\ref{thm:closure}:}

We intend to prove Theorem~\ref{thm:closure} by induction on $n$.

The case of $n=1$ is proved in Corollary~\ref{cor:n1}.

As an induction hypothesis, we assume that Theorem~\ref{thm:closure}
is valid for all $k$ in place of $n$ in its statement, where $k\leq
n-1$. In particular, the hypothesis of Theorem~\ref{thm:multi} is
satisfied.

Let $X=\cl{Ux}$. Let $V$ denote a maximal multi-parameter subgroup of
$W$ such that $Vx'\subset X$ for some $x'\in X$. Let $Z$ be a compact
$V$-minimal subset contained in $\cl{Vx'}$. Therefore by
Theorem~\ref{thm:DU-minimal} and by Theorem~\ref{thm:multi}, there
exists $\cJ\in\cC_0$ and $w\in U^\perp$ such that $Z=wH_{\cJ} w\inv
z'$, where $z'\in Z$ and $V\subset wH_{\cJ}w\inv$.

Note that $w\inv X=\cl{Uw\inv x}$. Now if we can show that $w\inv
X=H_{\cJ'}(w\inv x)$, then $X=wH_{\cJ'}w\inv x$ and the conclusion of
the theorem follows. Therefore without loss of generality, we replace
$X$ by $w\inv X$, $Z$ by $w\inv Z$, and $z'$ by $w\inv z'$, and assume
that $Z=H_{\cJ} z'$.

If $H_{\cJ}=G$, then $X=G/\Gamma$ and the theorem is proved.

Therefore we can assume that $H_{\cJ}\cong \SL_2(\qp)^m$ for some
$m\leq n-1$. In view of \eqref{eq:GJ_x}, we have
\begin{equation} \label{eq:HJ_z}
  \Lambda:=(H_\cJ)_{z'}=\prod_{J\in\cJ} H_J/(H_J)_{z'}.
\end{equation}
We note that
\begin{equation} \label{eq:H/Lam-Z}
  H_\cJ/\Lambda\cong H_{\cJ}z'=Z
\end{equation} 
is compact. Therefore $H_J/(H_J)_{z'}$ is compact for all $J\in\cJ$.
Therefore by Proposition~\ref{prop:generic} applied to $H_{\cJ}$ in
place of $G$ and $\Lambda$ in place of $\Gamma$,
\begin{equation} \label{eq:z:generic}
  \exists z\in
  H_{\cJ}/\Lambda \setminus \cS(U,\Lambda).
\end{equation}
In view of \eqref{eq:H/Lam-Z} we treat $z$ as an element of $Z$, and
hence
\begin{equation} \label{eq:HJz-Z}
  H_{\cJ} z'=Z=H_{\cJ}z\subset X.
\end{equation}
We have made such a choice of $z\not\in\cS(U,\Lambda)$ because later in
the proof we intend to apply Theorem~\ref{thm:avoid2} for the
$U$-action on $H_{\cJ}/\Lambda$.

We define
\begin{equation*}
  \cJ^\ast=\{J\setminus \max\{J\}:J\in \cJ,\,\num{J}>1\}.
\end{equation*}
Now $\cJ\in\cC_0$. Therefore $G=G_{\cJ^\ast}\cdot H_{\cJ}$.

Since $z\in Z\subset \cl{Ux}$, there exists a sequence $g_i\to e$ in
$G$ such that $g_iz\in Ux$ for all $i$. We can express $g_i=X_ih_i$
such that $X_i\in G_{\cJ^\ast}$, $h_i\in H_{\cJ}$, and $X_i\to e$ and
$h_i\to e$. 

If $X_{i_0}\in W$ for some $i_0$, then
\begin{equation} \label{eq:Xi:nor(HJ)}
  H_{\cJ} z\subset X = \cl{Ug_iz}=\cl{X_{i_0}Uh_iz}\subset
  X_{i_0}H_{\cJ}
  z=X_{i_0}H_{\cJ}X_{i_0}\inv(X_{i_0}z).
\end{equation}
In particular, $z$ belongs to the closed orbit
$(X_{i_0}H_{\cJ}X_{i_0}\inv)(X_{i_0}z)$. Therefore
\begin{equation*}
  H_{\cJ}z\subset
  X_{i_0}H_{\cJ}X_{i_0}\inv(X_{i_0}z)=(X_{i_0}H_{\cJ}X_{i_0}\inv) z.
\end{equation*}
Hence $H_{\cJ}$ is an open subgroup of $X_{i_0}H_{\cJ}X_{i_0}\inv$.
Since $H_{\cJ}$ is Zariski closed, we have that
$H_{\cJ}=X_{i_0}H_{\cJ}X_{i_0}\inv$. Therefore the inclusions in
\eqref{eq:Xi:nor(HJ)} are equalities. Hence $X=H_{\cJ} z$, and the
conclusion of the theorem holds.

Now we may assume that $\{X_i\}\subset G_{\cJ^\ast}\setminus W$.  Put
$m=\num{\cup\cJ^\ast}$. In view of the identification,
$G_{\cJ^\ast}\cong \SL_2(\qp)^m$, we have that
\begin{equation*}
  \{X_i\}\subset M_2(\qp)^m\setminus L_m
\end{equation*}
(recall \eqref{eq:E_m}). Also the conjugation action of $u(t)$ on
$G_{\cJ^\ast}$ corresponds to the conjugation action of $w_m(t)$ on
$M_2(\qp)^m$. Therefore by Lemma~\ref{lema:conj:u}, there exists a
sequence $t_i\to\infty$ and a non-constant polynomial map $\psi:\qp\to
\qp^n$ such that for any sequence $s_i\to s$ in $\qp$,
\begin{equation} \label{eq:sti}
  \lim_{i\to\infty} u(s_it_i)X_iu(-s_it_i)=w(\psi(s))\in W_{\cJ^\ast}.
\end{equation}

If $Z$ were $U$-minimal, which would be the case if $H_\cJ\cong
\SL_2(K)$, or if $n=2$ and $m\leq n-1=1$. We would then apply
Lemma~\ref{lema:top1} for $Y'=X$, $Y=Z$, $P'=U$, $P=H_\cJ$ and $F=U$;
and conclude that $\Psi(s)X\subset X$. 

In general, we will have to go deeper into the proof of
Lemma~\ref{lema:top1} to see what is exactly required; and that turns
out to be Theorem~\ref{thm:avoid2} as shown below. 

In view of \eqref{eq:HJ_z} and \eqref{eq:HJz-Z}, we apply
Theorem~\ref{thm:avoid2} to $H_{\cJ}$ and $\Lambda$ in places of $G$
and $\Gamma$, respectively, and to the sequence
$\{x_i:=h_iz\}_{i\in\N}\subset Z$. Since $x_i\to z$ and $z\not\in
\cS(U,\Lambda)$ (see \eqref{eq:z:generic}), we conclude the following:
given any compact neighbourhood $\frakO$ of $0$ in $\qp$ and $s\in
\qp$, there exists a sequence $t_i'\in st_i(1+\frakO)$ such that,
after passing to a subsequence, $u(t_i')x_i\to y$ as $i\to\infty$,
where $y\in Z\cong H_{\cJ}/\Lambda$ and $y\not\in \cS(U,\Lambda)$.

Since $H_{\cJ}\cong \SL_2(\qp)^m$ for some $m\leq n-1$, by our
induction hypothesis, Theorem~\ref{thm:closure} is valid for $H_{\cJ}$
in place of $G$. Therefore, since $y$ is nonsingular for the $U$
action on $Z$, we conclude that
\begin{equation}
  \label{eq:Uy}
  \cl{Uy}=Z.
\end{equation}
Note that this is the second instance of the use of the induction
hypothesis in this proof.

We put $s_i=t_i'/t_i\in s(1+\frakO)$ for all $i$. Then $t_i'=s_it_i$,
and after passing to a subsequence, we may assume that $s_i\to s'$ and
$s'\in s(1+\frakO)$. Now by \eqref{eq:sti},
\begin{equation*}
  \begin{array}{ll}
    u(s_it_i)g_iz & =u(s_it_i)X_ix_i \\
    { } & =[u(s_it_i)X_iu(-s_it_i)]u(s_it_i)x_i \\
    { }           & \to w(\psi(s'))y.
  \end{array}
\end{equation*}
Thus $w(\psi(s'))y\in X$, and hence by \eqref{eq:Uy}
\begin{equation*}
  X\supset \cl{Uw(\psi(s'))y}=w(\psi(s'))\cl{Uy} = \Psi(s')Z.
\end{equation*}
Since $\frakO$ was an arbitrarily chosen neighbourhood of $0$, and $s'\in
s(1+\frakO)$, we conclude that
\begin{equation} \label{eq:X-wZ}
  X\supset w(\psi(s))Z,\qquad \forall s\in\qp.
\end{equation}

This finishes a major step in the proof, as we have obtained a
nontrivial trajectory of a polynomial set in $W_{\cJ^\ast}$. Now we
will use an idea from \cite{Dani+Mar:Oppen-Invent} to show that $X$
contains a trajectory of a nontrivial one-parameter subgroup of
$W_{\cJ^\ast}$.

Since $X$ is compact, there exists a sequence $T_i\to\infty$ in $\qp$
and $x'\in X$ such that
\begin{equation}
  \label{eq:x'}
  w(\psi(T_i))z\to x'.
\end{equation}
Then by Lemma~\ref{st-q},
\begin{equation*}
  \psi(T_i+sT_i^{-q})-\psi(T_i)\to s\vv, \qquad \forall s\in\qp,
\end{equation*}
where $q=\deg(\psi)-1\in\{0,1\}$ and $\vv\in\qp^n\setminus\{0\}$.
Therefore by \eqref{eq:x'}
\begin{equation*}
  w(\psi(T_i+sT_i^{-q}))z = w(\psi(T_i+sT_i^{-q}-\psi(T_i))w(\psi(T_i))z
  \to w(s\vv)x'.
\end{equation*}
Therefore, since $VZ=Z$, for any $u\in V$, by \eqref{eq:X-wZ},
\begin{equation*}
  X\ni w(\psi(T_i+sT_i^{-q}))uz=uw(\psi(T_i+sT_i^{-q}))z\to
  uw(s\vv)x'.
\end{equation*}
Thus $VV_1x'\subset X$, where $V_1=\{w(s\vv):s\in\qp\}$. We note that
$V\subset H_\cJ$ and $\psi(s)\in W_{\cJ^\ast}$. Therefore $V_1$ is a
nontrivial one-parameter subgroup of $W_{\cJ^\ast}$, which is not
contained in $V$. Thus $VV_1$ is a multi-parameter subgroup of $W$
which is strictly larger than $V$, and $VV_1x'\subset X$. This
contradicts the maximality property of $V$ assumed at the beginning of
the proof. This completes the proof of the theorem.
\qed

\section{$H$-orbit closures}

\begin{lema} \label{lema:wHw}
  If $D\subset wH_{\cJ}w\inv$ for some $w\in W$ and $\cJ\in\cC_0$ then
  $w\in H_{\cJ}$. 
\end{lema}

\begin{proof}
  It easily follows from the facts that $\nor(H_{\cJ})=Z(G)H_{\cJ}$, and that
  $d(a)w(\vt)d(a)\inv=w(a^2\vt)$ for any $t\in\qp^n$ and
  $a\in\qp^\ast$.
\end{proof}

Define $\frakF$ to be the collection of closed subgroups $F$ of $G$
with the following properties: $F/F\cap\Gamma$ is compact, and 
$F=gH_{\cJ}g\inv$ for some $g\in G$ and $\cJ\in\cC$.

\begin{lema} \label{lema:countable}
  $\frakF$ is countable.
\end{lema}

\begin{proof}
  Let $F\in\frakF$. In view of Remark~\ref{rem:prod}, 
  \begin{equation*}
  F/F\cap\Gamma\cong \prod_{i=1}^r \SL_2(\qp)/\Lambda_i,
  \end{equation*}
  where $\Lambda_i$ is a cocompact discrete subgroup of $\SL_2(\qp)$ and
  $1\leq r\leq n$. It is straightforward to verify that each $\Lambda_i$
  is Zariski dense in $\SL_2(\qp)$ (this is a very special easy case of
  the Borel's density theorem (see~\cite{Dani:simple,Dani:proj-lin} or
  \cite{Rag:book}. Therefore $\Zcl(F\cap\Gamma)=F$, where $\Zcl(X)$
  denotes the Zariski closure of a set $X$ in $\M_2(\qp)^n$. Now there
  exists a finite set $S\subset F\cap\Gamma$ such that if $\langle
  S\rangle$ denotes the subgroup generated by $S$ then
  \begin{equation*}
  \Zcl(\langle S \rangle)=\Zcl(F\cap\Gamma)=F.
  \end{equation*}
  
  Thus
  \begin{equation*}
  \frakF\subset \{\Zcl(\langle S \rangle):\mbox{$S$ is a finite subset of
    $\Gamma$}\}. 
  \end{equation*}
  Since $\Gamma$ is countable, $\frakF$ is countable.
\end{proof}

\begin{proof}[Proof of Corollary~\ref{cor:H}]

  For any $h\in H$, by Theorem~\ref{thm:closure}, there exist $w\in W$
  and $\cJ\in\cC_0$ such that
  \begin{equation*}
  \cl{hUh\inv}x=h\cl{U(h\inv x)}=h(wH_{\cJ} w\inv) (h\inv x)=F_hx,
  \end{equation*}
  where $F_h:=hwH_{\cJ} w\inv h\inv$.
  
  Suppose if $H\subset F_h$ then $H\subset wH_{\cJ}w\inv$, and by
  Lemma~\ref{lema:wHw}, we have $w\in H_{\cJ}$ and $F_h=H_{\cJ}$.
  Hence $H_{\cJ}x$ is compact, and
  \begin{equation*}
  H_{\cJ} x\supset \cl{Hx}\supset (hUh\inv)x=H_{\cJ} x.
  \end{equation*}
  Thus $\cl{Hx}=H_{\cJ} x$, and we are through.
 
  Suppose that $H\not\subset F_h$, then $hUh\inv \subset F_h\cap H$,
  which is a proper algebraic subgroup of $H\cong \SL_2(K)$. Therefore
  $F_h\cap H$ at most $2$ dimensional, and any nontrivial algebraic
  unipotent subgroup of $F_h\cap H$ equals $hUh\inv$. Hence for any
  $h_1\in H$, if $F_{h_1}=F_h$ then $h_1 Uh_1\inv=hUh\inv$. Thus, 
  \begin{equation} \label{eq:h1:hNU} \text{for any $h,h_1\in H$: if
      $H\not\subset F_h$ and $F_h=F_{h_1}$, then $h_1\in hN_H(U)$}.
\end{equation}

Now fix $g\in G$ such that $x=g[\Gamma]\in G/\Gamma$. Since $F_hx$ is
compact, we have $gF_hx=gF_hg\inv\Gamma/\Gamma$ is compact.  Therefore
$gF_hg\inv\in \frakF$. Since $\frakF$ is countable, the collection
$\{F_h:h\in H\}$ is countable. Hence due to \eqref{eq:h1:hNU}, since
$H/N_H(U)$ is uncountable, there exists $h\in H$ such that $F_h\supset
H$, and we are back to the case considered earlier.
\end{proof}

\subsection*{Proof of Proposition~\ref{prop:H:comm}}
Since $H_J=G_J\cap H_{\cJ}$, $Y:=G_Jx_0\cap H_{\cJ} x_0$ is compact,
and the stabilizer of $x_0$, which is $\Gamma$, is discrete, we
conclude that every orbit of $H_J$ in $Y$ is open. Therefore every
orbit of $H_J$ in $Y$ is closed. In particular, $H_Jx_0$ is compact.

Therefore replacing $G$ by $G_J$, $H_{\cJ}$ by $H_J$, and $\Gamma$ by
$G_J\cap\Gamma$, without loss of generality we may assume that $Hx_0$
is compact.
  
  In view of Remark~\ref{rem:prod}, we define the natural projection
  maps $q_j:G\to G_{\{j\}}$ and $\bar q_j:G/\Gamma\to
  G_{\{j\}}/\Gamma_j$. Now ${\bar q_j}\inv(e\Gamma_j)\cap Hx_0$ is a
  compact subset of $G/\Gamma$. Since it is countable, it is
  finite. Therefore 
  \begin{equation*}
  {\bar q_j}\inv(e\Gamma_j)\cap Hx_0 \cong q_j\inv(\Gamma_j)\cap
  H/q_j\inv(\Gamma_j)\cap H\cap \Gamma
  \end{equation*}
  is finite. Now 
  \begin{equation*}
  q_j\inv(\Gamma_j)\cap H=\{(\gamma,\ldots,\gamma)\in G:\gamma\in
  \Gamma_j\}
  \end{equation*} and 
  \begin{equation*}
  q_j\inv(\Gamma_j)\cap H\cap\Gamma=\{(\gamma,\ldots,\gamma)\in G:
  \gamma\in \cap_{i=1}^n \Gamma_i\}.
  \end{equation*}
  Therefore $\cap_{i=1}^n\Gamma_i$ is a subgroup of finite index in
  $\Gamma_j$. Therefore $\Gamma_i$ and $\Gamma_j$ are commensurable for
  all $i$ and $j$. 
\qed

\subsection*{Proof of Corollary~\ref{cor:H:closure} }
Let $J\in\cJ$, and $\Lambda_J=\cap_{j\in J}p_j(\Gamma_j)$.  Then by
definition $\Lambda_J$ is a subgroup of finite index in
$p_j(\Gamma_j)$ for each $j\in J$, and hence $\Lambda_J$ is a
cocompact lattice in $\SL_2(K)$. Clearly, $H_J/(H_J\cap\Gamma)\cong
SL_2(K)/\Lambda_J$ is compact. Therefore $H_Jx_0$ is compact.

  From this we obtain that $H_{\cJ}x_0$ is compact. Now for any
  $J_1,J_2\in\cJ$ with $J_1\neq J_2$, we have that the lattices
  $\Lambda_{J_1}$ and $\Lambda_{J_2}$ are noncommensurable. Therefore
  applying Corollary~\ref{cor:H:noncomm} to $H_{\cJ}$ in place of $G$,
  we conclude that $Hx_0$ is dense in $H_{\cJ}x_0$.
\qed

\section{Limiting distributions of sequences of unipotent orbits}

As noted in the introduction, we start the second half of the article.
First we give the statement of the main result, which says that a
unipotent trajectory starting from a non-singular point attaches zero
measure on its singular set $\cS(U,\Gamma)$ in the limiting
distribution. 

\subsubsection*{Notation} 
Let $\cM=\cM(G/\Gamma)$ denote the space of probability measures on
$G/\Gamma$, which is compact. Then $\cM$ is compact with respect to
the topology of weak-$\ast$ convergence; here by definition, a
sequence $\mu_i\to\mu$ in $\cM$ if $\int f\,d\mu_i\to\int f\,d\mu$ as
$i\to\infty$, for all $f\in C(G/\Gamma)$.

Let $\theta$ denote a Haar measure on $\qp$. 

\begin{theo} \label{thm:avoid}
  Let $x_i\to x$ be a sequence in $G/\Gamma$ and $t_i\to\infty$ be a
  sequence in $\qp$. Fix any measurable set $\frakO\subset \qp$ with
  $0<\theta(\frakO)<\infty$. Let $\mu_i=\mu^\frakO_i\in\cM(G/\Gamma)$
  be defined as
  \begin{equation} \label{eq:mu_i}
    \mu^\frakO_i(E)=
    \frac{\theta(\{t\in t_i\frakO:u(t)x_i\in E\})}{\theta(t_i\frakO)}, 
    \quad \text{for all Borel sets $E\subset G/\Gamma$}.
  \end{equation}
  Let $\mu\in\cM$ be a limit of any subsequence of
  $\{\mu_i\}_{i=1}^\infty$ in $\cM$.  Further suppose that $x\not\in
  \cS(U,\Gamma)$. Then $\mu(\cS(U,\Gamma))=0$.
\end{theo} 

As a first consequence of this result, we deduce the result required
in the proof of Theorem~\ref{thm:closure}.

\subsection{Proof of Theorem~\ref{thm:avoid2}.}
\label{subsec:thm:avoid2} Given a compact neighbourhood $\frakO$ of
$0$ in $K$, we apply Theorem~\ref{thm:avoid} for $1+\frakO$ in place
of $\frakO$ in the statement above. Since
$\mu(\cS(U,\Gamma)=\emptyset$, we can choose $y\in \supp(\mu)\setminus
\cS(U,\Gamma)$. Let $\Omega_i$ be a sequence of open neighbourhoods of
$y$ in $G/\Gamma$ such that $\cap_i\Omega_i=\{y\}$.  Now by the
definition of $\mu_i=\mu_i^{1+\frakO}$, by passing to a subsequence of
$i$, we may assume that
$\supp(\mu_i^{1+\frakO})\cap\Omega_i\neq\emptyset$. Then there exists
$t_i'\in (1+\frakO)t_i$ such that $u(t_i')x_i\in \Omega_i$. Therefore
$u(t_i')x_i\to y$ as $i\to\infty$. \qed

\subsection{Uniform distribution of $U$-orbits}

As another main consequence of Theorem~\ref{thm:avoid} we will deduce
the uniform distribution of $U$-orbits using Ratner's measure
classification result. We first give an idea of the connection of both
the results.

\begin{lema} \label{lema:inv} Any limit measure $\mu$ as obtained in
  Theorem~\ref{thm:avoid} is $U$-invariant.
\end{lema} 

Since invariant measures decompose into its ergodic components, using
the description of ergodic $U$-invariant measures
\cite{R:p-adic,Mar+Tom:Invent} and Theorem~\ref{thm:avoid}, we will
obtain the following uniform distribution result.

\begin{theo} \label{thm:uniform}
  Let $\frakO$ be a measurable subset of $\qp$ such that
  $0<\theta(\frakO)<\infty$. Fix any $x\in G/\Gamma$ then there exists
  $w\in W$ and $\cJ\in\cC$ such that $\cl{Ux}=wH_{\cJ}w\inv x$ and the
  following holds: For $T\in\qp\setminus\{0\}$ define $\mu_T\in\cM$ as
  \begin{equation}
    \label{eq:mu_Ti}
    \mu_T(E)=\frac{\theta(\{t\in T\frakO: u(t)x\in E\})}{\theta(T\frakO)},
    \quad
    \text{for all Borel sets $E\subset G/\Gamma$}.
  \end{equation}
  Then for any continuous function $f$ on $G/\Gamma$, we have
  \begin{equation*}
    \int f\, d\mu_T \to \int f\,d\mu \quad \text{as $T\to\infty$ in $\qp$,} 
  \end{equation*}
  where $\mu$ denotes the unique $wH_{\cJ}w\inv$-invariant probability
  measure on the space
  \begin{equation*}
    wH_{\cJ}w\inv x\cong wH_{\cJ}w\inv/(wH_{\cJ}w\inv\cap G_x),
  \end{equation*}
  where $G_x$ denotes the stabilizer of $x$ in $G$. 
\end{theo}

\section{A countability theorem and the singular set}

Note that for any $g\in G$, and $x=gx_0$, the orbit
$G_jx=gG_{\{j\}}x_0$ is compact for any $j=1,\dots,n$, where $x_0\in
G/\Gamma$ denotes the coset of the identity. Similarly, $G_J x$ is
compact for any nonempty $J\subset\{1,\dots,n\}$.
  
Let $\Gamma_J=G_J\cap\Gamma$, and $\bar\rho_J:G/\Gamma \to
G_J/\Gamma_J$ denotes the natural projection in view of
\eqref{eq:Gamma-prod}. Note that every fiber of $\bar\rho_J$ is a
compact orbit of the group $G_{J^c}$, where
$J^c=\{1,\ldots,n\}\setminus J$. Therefore $\bar\rho_J$ is a proper
map; namely, the inverse images of compact sets are compact. 
 
We assume that $n\geq 2$. Let $\cH$ denote the collection of all
subgroups $F$ of $G$ with the following properties: (i) $F/F\cap
\Gamma$ is compact, and (ii) $F=f\inv G_{J^c}H_Jf$ for some $f\in
G_J$, where $J\subset \{1,\ldots,n\}$, $\abs{J}=2$.  Note that $Z(G)F$
is a proper maximal subgroup of $G$, where $Z(G)=\{(\pm I,\dots,\pm
I)\}$ denotes the center of $G$, and $I=\smat{1& \\ &1}$.

Note that $F\cap\Gamma=\Gamma_{J^c}(fH_{J}f\inv \cap \Gamma_J)$. Let
$\Lambda=fH_Jf\inv\cap\Gamma_J$. Then $fH_Jf\inv/\Lambda$ is compact
and admits an $fH_Jf\inv$-invariant probability measure. This measure
projects onto an $fH_Jf\inv$-probability measure on $fH_Jf\inv/L$,
where $L$ denotes the Zariski closure of $\Lambda$ in $fH_Jf\inv$.
Since $H_J\cong \SL_2(\qp)$, if $L$ is one dimensional then the
quotient cannot be compact, and if $L$ is two dimensional then the
quotient is a projective line and does not admit an invariant measure.
Therefore $L=fH_Jf\inv$; we remark that this conclusion is also a
special case of Borel's density theorem~\cite{Rag:book,Dani:simple}.
Therefore $fH_{J}f\inv$ is the Zariski closure of the subgroup
generated by a finite subset of $\Gamma_J$. Hence $F\in\cH$ is
determined by $J$ and a finite subset of $\Gamma$. Since $\Gamma$ is
countable, we conclude the following:

\begin{lema} \label{lema:cH}
  The collection $\cH$ is countable.
\end{lema}

For any $F\in \cH$, we define (the algebraic variety)
\begin{equation*}
  X(F)=\{g\in G: U\subset gFg\inv\}.
\end{equation*}
Note that for any $F\in\cH$ and $g\in G$: 
\begin{equation} \label{eq:XF}
  g\in X(F) \Leftrightarrow 
  \cl{Ugx_0}\subset gFx_0=(gFg\inv )gx_0,
\end{equation}
where $x_0=\pi(e)$ and $\pi:G\to G/\Gamma$ is the natural quotient
map. 

\begin{lema} \label{lema:S}
  $\cS(G/\Gamma)=\bigcup_{F\in\cH} \pi(X(F))$.
\end{lema}

\begin{proof}
  By (\ref{eq:XF}), $\pi(X(F))\subset \cS(G/\Gamma)$. 
  
  Now let $g\in G$ such that $gx_0\in\cS(G/\Gamma)$. Then there exists
  $\cJ\in\cC$ and $w\in W$ such that $\cup\cJ=\{1,\dots,n\}$,
  $H_\cJ\neq G$, $U \subset wH_{\cJ}w\inv$ and $H_{\cJ}w\inv g x_0$ is
  compact.  
  
  Therefore there exists $1\leq j_1<j_2\leq n$ such that, if
  $g=(g_1,\ldots,g_n)$ and $g\in H_{\cJ}$, then $g_{j_1}=g_{j_2}$. Put
  $J=\{j_1,j_2\}$. Since $G=G_{J^c}G_J$, there exists $f\in G_J$ such
  that $g\inv w\in G_{J^c}f$.
  
  If we put $F=G_{J^c}fH_J f\inv$, then $F=G_{J^c}(g\inv
  w)H_{\cJ}(w\inv g)$. Since $G_{J^c}z$ is compact for all $z\in
  G/\Gamma$, $Fx_0$ is compact. Hence $\cl{g\inv U gx_0} \subset
  Fx_0$. Therefore $g\inv U g\subset F$, and hence $g\in X(F)$.
\end{proof}

\begin{lema} \label{lema:X(F)}
  Let $F\in\cH$, $J=\{j_1,j_2\}$, $1\leq j_1<j_2\leq n$, and $f\in G_J$
  such that $F=f\inv G_{J^c}H_J f$.  Then $X(F)=WG_{J^c}H_{J}Z(G)f$.
  Moreover $H_J(zfx_0)$ is compact for every $z\in Z(G)$. 
\end{lema}

\begin{proof} 
  Take any $g\in X(F)$. Let $U_J=U\cap H_J$. Then $g\inv U_J g\subset
  f\inv H_{J}f$. Since $H_J\cong \SL_2(\qp)$, there exists $h\in H$
  such that
  \begin{equation*}
    g\inv U_J g = f\inv hU_{J}h\inv f.
  \end{equation*}
  Therefore $h\inv fg\in D_JW_JZ(G)G_{J^c}$. Multiplying $h$ by an
  appropriate element of $D_J$ on the right, we may assume that $h\inv
  fg\in W_JG_{J^c}Z(G)$. Hence $g\in hfW_JG_{J^c}Z(G)\subset
  W_JG_{J^c}H_{J}Z(G)f$. 
  
  Moreover $G_{J^c}H_J(zfx_0)=zfG_{J^c}(f\inv H_\cJ f)x_0=zfFx_0$ is
  compact.  
\end{proof}

\subsection{Proof of Proposition~\ref{prop:generic}:}
\label{subsec:generic} 
By Lemma~\ref{lema:X(F)}, the set $X(F)\gamma$ cannot contain an open
subset of $G$ for any $F\in\cH$ and $\gamma\in\Gamma$. Now $X(F)$ can
be expressed as a countable union of compact sets, and since $\cH$ and
$\Gamma$ are countable sets, by Baire's category theorem we have that
$G\neq \bigcup_{F\in\cH} X(F)\Gamma$. Therefore $G/\Gamma\neq
\cS(U,\Gamma)$ by Lemma~\ref{lema:S}.  \qed

\section{Reducing Theorem~\ref{thm:avoid} to the case of $n=2$}
\label{sec:reduction}

By Lemma~\ref{lema:S} and Lemma~\ref{lema:X(F)}, in order to prove
that $\mu(\cS(G/\Gamma))=0$, it is enough to show that
$\mu(WG_{J^c}H_Jy)=0$ for every $J=\{j_1,j_2\}$, $1\leq j_1 < j_2\leq
n$, and $y\in G/\Gamma$ such that $H_Jy$ is compact.

Fix $J$ and $y$ as above. Then $H_J\bar y$ is compact in
$G_J/\Gamma_J$, where $\bar y=\bar\rho_J(y)$. Also
\begin{equation} \label{eq:z:J}
  (\bar\rho_J)\inv(W_JH_J\bar y)=WG_{J^c}H_Jy. 
\end{equation}

Let $\bar\mu$ denote the projection of $\mu$ on $G_J/\Gamma_J$ via
$\bar\rho_J$; that is, $\bar\mu(E)=\mu((\bar\rho_J)\inv(E))$ for any
Borel measurable set $E\subset G_J/\Gamma_J$. Therefore in order to
prove that $\mu(WG_{J^c}H_J y)=0$, it is enough to show that
$\bar\mu(W_JH_J\bar y)=0$. Further it is enough to show that for any
compact set $C\subset W_J$,
\begin{equation} \label{eq:enough}
  \bar\mu(CH_J\bar y)=0.
\end{equation}

Note that $G_J\cong \SL_2(\qp)\times \SL_2(\qp)$, and under this
isomorphism $H_J$ corresponds to the diagonally embedded copy of
$\SL_2(\qp)$ in $\SL_2(\qp)\times\SL_2(\qp)$. For the projection
homomorphism $\rho_J:G\to G_J$, let $\bar u(t):=\rho_J(u(t))\in H_J$
for all $t\in\qp$. Let $\bar{x}_i=\bar\rho_J(x_i)$, and let
$\bar\mu_i\in \cM(G_J/\Gamma_J)$ be such that
\begin{equation*}
  \bar\mu_i(E) = 
  \frac{\theta(\{t\in t_i\frakO:\bar u(t)\bar x_i\in E\})}
  {\theta(t_i\frakO)},
  \quad 
  \text{for all Borel sets $E\subset G_J/\Gamma_J$}.
\end{equation*}
Then $\bar\mu_i$ is the projection of $\mu_i$ on $G_J/\Gamma_J$.
Furthermore whenever $\mu_i\to\mu$ in $\cM(G/\Gamma)$, we have
$\bar\mu_i\to \bar\mu$. Since $x\not\in
WG_{J^c}H_Jy\subset\cS(U,\Gamma)$, by \eqref{eq:z:J}, $\bar
x:=\bar\rho_J(x)\not\in W_JH_J\bar y$, and $\bar x_i\to\bar x$.

In view of the above explanation, to prove Theorem~\ref{thm:avoid} it
is enough to prove it for the case of $n=2$. 

For $r>0$, and $x\in\qp$, let $B_x(r)$ denote the ball of radius $r>0$
in $\qp$ centered at $x$. 

\subsection{Reduction to the case of $\frakO=B_0(r)$}
Since $0<\theta(\frakO)<\infty$, given any $\beta<1$ there exists a
compact subset $\frakO_1\subset\frakO$ such that
$\theta(\frakO_1)/\theta(\frakO)>\beta$.  Therefore it will be enough
to prove the result under the assumption that $\frakO\subset B_0(r)$
for some $r>0$. Put $B=B_0(r)$.

Let $\lambda_i=\mu^B$ as defined in \eqref{eq:mu_i}. Then
$\mu_i(E)\leq (\theta(B)/\theta(\frakO))\lambda_i(E)$ for any Borel
set $E\subset G/\Gamma$. By passing to a subsequences we have that
$\mu_i\to\mu$ and $\lambda_i\to\lambda$ as $i\to\infty$.  Therefore
$\mu(E)\leq (\theta(B)/\theta(\frakO))\lambda(E)$ for all Borel sets
$E\subset G/\Gamma$.  Therefore if we prove that
$\lambda(\cS(U,\Gamma))=0$, then $\mu(\cS(U,\Gamma))=0$. This proves
that it is enough to prove Theorem~\ref{thm:avoid} for $\frakO=B_0(r)$
for all $r>0$.

\section{Theorem~\ref{thm:avoid} for $G=\SL_2(\qp)\times \SL_2(\qp)$}
\label{sec:n:2}

Let $G=\SL_2(\qp)\times \SL_2(\qp)$ and $H$ be the diagonal embedding
of $\SL_2(\qp)$ in $G$. For $t_1,t_2\in\qp$, let $w(t_1,t_2):=
\bigl(\smat{1&t_1\\&1},\smat{1&t_2\\&1}\bigr)$. Let
$W=\{w(t_1,t_2):t_i\in\qp$. Define $u(t)=w(t,t)\in G$, $\forall
t\in\qp$, and $U=\{u(t):t\in\qp\}=W\cap H$.

Let $\Gamma$ be a discrete subgroup of $G$ such that $G/\Gamma$ is
compact.

In this section we will prove the following: 

\begin{theo} \label{thm:avoid:2}
  Let $y\in G/\Gamma$ such that $Hy$ is compact. Let $x_i\to x$ be a
  convergent sequence in $G/\Gamma$ such that $x\not\in WHy$. Then
  given any $\epsilon>0$ and a compact set $C_1\subset W$ there exist
  a neighbourhood $\Psi_1$ of $C_1Hy$ in $G/\Gamma$ and a natural
  number $i_0$ such that $\forall i\geq i_0$ and $T>0$,
\begin{equation} \label{eq:avoid:2}
  \theta(\{t\in B_0(T):u(t)x_i\in \Psi_1\})\leq \epsilon\theta(B_0(T)).
\end{equation}
\end{theo}

\subsection{Proof of Theorem~\ref{thm:avoid}}

Let $\frakO=B_0(r)$ for some $r>0$. In view of~\eqref{eq:mu_i},
$\mu_i(\Psi_1)\leq\epsilon\cdot\theta(B_0(r))$ for all $i\geq i_0$,
and hence $\mu(C_1Hy)=0$. Since $C_1$ can be chosen to be an arbitrary
compact subset of $W$, we have that $\mu(WHy)=0$. Thus in view of the
discussion in Section~\ref{sec:reduction}, the
Theorem~\ref{thm:avoid:2} implies Theorem~\ref{thm:avoid}.  \qed

\subsection{Linearization of the $U$-action near $WHy$}

For a group $F$ acting on a set $X$ and an element $x\in X$, let
\[
F_x=\{f\in F:fx=x\}, \quad \text{the stabilizer of $x$ in $F$}.
\]

Note that $G=G_{\{1\}}H$ and $WH=W_1H$, where $W_1=G_{\{1\}}\cap
W=\{w(t,0):t\in\qp\}$. Let $I=\smat{1&0\\0&1}$. 

\begin{lema} \label{lema:wHw2}
$wHw\inv \cap H=U\cup (-I,-I)U$ for all $w\in W_1\setminus\{e\}$. 
\end{lema}

\begin{proof} Let $h=(x,x)\in H$ and $w=(w_1,I)\in W_1$, $w_1\neq I$. Then
  $w h w\inv \in H\Rightarrow x=w_1xw_1\inv \Rightarrow x=\smat{\pm
    1&s\\0&\pm 1}$, $s\in\qp$.
\end{proof}

The next observation, which states that the singular set $WHy=W_1Hy$
does not self-intersect along $W_1$, makes the study of dynamics near
singular sets much simpler in our situation, as compared to the
general case~\cite[Lemma~6.5]{Shah:uniform}.

\begin{prop} \label{prop:disjoint}
  For $w_1,w_2\in W_1$, if $w_1\neq w_2$ then $w_1Hy\cap
 w_2Hy=\emptyset$.
\end{prop}

\begin{proof} 
  Let $Z=w_1Hy\cap w_2Hy$. Suppose that $Z\neq\emptyset$.  Put
  $H_i=w_iHw_i\inv$. Then $w_iHy=H_1(w_iy)=H_iz$ is compact for every
  $z\in Z$. Since $G_z$ is a discrete group, $(H_1\cap H_2)z$ is open
  in $Z=H_1z\cap H_2z$. Since $w_1\neq w_2$, by Lemma~\ref{lema:wHw2},
  $U$ is an open subgroup of $H_1\cap H_2$. Therefore every orbit of
  $U$ on $Z$ is open in $Z$. Hence every orbit of $U$ on $Z$ is
  closed. Since $Z$ is compact, $Uz\cong U/U\cap G_z$ is compact,
  which contradicts Proposition~\ref{prop:cocomp}.
\end{proof}

We consider a linear action of $G$ on $E:=M_2(\qp)$ defined as
follows: Given $g=(g_1,g_2)\in G$ and $X\in E$,
\begin{equation*}
  g\cdot X:=g_1Xg_2\inv.
\end{equation*} 
Let $I=\smat{1&0\\ 0&1}\in E$. Then
\begin{align} \label{eq:stab:I}
  H & =\{g\in G: g\cdot I=I\} \\
  G\cdot I & = \SL_2(\qp)\subset E.
\label{eq:GI}
\end{align}
Let $\cW=\{w_1(t):t\in\qp\}\subset E$. Then $W_1\cdot I=\cW$, and 
\begin{equation} \label{eq:cW} 
W_1H=WH =\{g\in G: g\cdot I\in\cW\}.
\end{equation}

\begin{lema} \label{lema:discrete} 
The set $G_y\cdot I$ is discrete.
\end{lema}  

\begin{proof}
  Since $Hy$ is compact, $H/H\cap G_y$ is compact, and hence $HG_y$ is
  closed in $G$. Therefore $HG_y$ is closed in $G$. Hence
  $G_yH=(HG_y)\inv$ is closed in $G$. Due to \eqref{eq:GI} and
  \eqref{eq:stab:I}, the map $G/H\to \SL_2(K)$ given by $gH\mapsto
  g\cdot I$ is a homeomorphism. Hence $G_y\cdot I$ is a closed subset
  of $\SL_2(K)$, and hence of $E$. Further since $G_y$ is countable,
  $G_y\cdot I$ is discrete.
\end{proof}

For any $z\in G/\Gamma$, we define $\cR(z)=\{g\cdot I:gz=y,\,g\in
G\}$. Note that if $z=gy$, then $\cR(z)=gG_y\cdot I=g\cR(y)$. The set
$\cR(z)$ is called the set of representatives of $z$ in $E$. By
Lemma~\ref{lema:discrete}, $\cR(z)$ is discrete.

\begin{lema} \label{lema:R(z):W}
$\card{\cR(z)\cap \cW}\leq 1$, for all $z\in G/\Gamma$.
\end{lema}

\begin{proof}
  If $g\gamma_1\cdot I,g\gamma_2\cdot I\in \cW$ for some
  $\gamma_1,\gamma_2\in G_y$, then by \eqref{eq:stab:I}, there exist
  $w_i\in W_1$ such that $g_i\gamma_i\in w_iH$ for $i=1,2$. Then
  \[
  g\gamma_1y=g\gamma_2y\in w_1Hy\cap w_2Hy.
  \]
  Therefore by Proposition~\ref{prop:disjoint}, $w_1=w_2$. Hence
  $g\gamma_1H=w_1 H=w_2H=g\gamma_2H$. Thus $g\gamma_1\cdot
  I=g\gamma_2\cdot I$.
\end{proof}

The following observation will allow us to `linearize' the $G$-action
in thin neighbourhoods of compact subsets of $WHy$.

\begin{lema} \label{lema:inject} 
  Given a compact subset $D$ of $\cW$, there exists a
  neighbourhood $\Phi$ of $D$ in $E$ such that
  $\card{\cR(z)\cap\Phi}\leq 1$ for all $z\in G/\Gamma$.
\end{lema}

\begin{proof} 
  Let $\{\Phi_i\}$ be a decreasing sequence of relatively compact
  neighbourhoods of $D$ in $E$ such that $\cap_i\,\Phi_i=D$. If the
  lemma is false, then there exists a sequence $\{z_i\}\subset
  G/\Gamma$ such that $\card{\cR(z_i)\cap\Phi_i}\geq 2$ for all $i$.
  By passing to a subsequence we may assume that $z_i=g_iy$ for a
  sequence $g_i\to g$ in $G$, and for each $i$ there exist
  $\gamma_i,\delta_i\in G_y$ such that 
  \begin{equation} \label{eq:gamma-delta}
    g_i\gamma_i\cdot I,g_i\cdot\delta_i\cdot I\in \Phi_i 
    \quad \text{and}\quad    
    \gamma_i\cdot I\neq \delta_i\cdot I.
  \end{equation}  
  Now 
  \begin{equation*}
\{\gamma_i,\delta_i:i\in\N\}\subset \cup_{i=1}^\infty \{g_i\inv
\Phi_i\}\subset(\{g_i:i\in\N\}\cup\{g\})\cl{\Phi_1},
\end{equation*}
which is compact. Therefore by Lemma~\ref{lema:discrete} there exist
$\gamma,\delta\in G_y$ such that $\gamma_i\cdot I=\gamma\cdot I$ and
$\delta_i\cdot I=\delta\cdot I$ for all large $i$. Therefore
$g_i\cdot\gamma\cdot I\to g\gamma\cdot I\in D$, and similarly
$g\delta\cdot I\in D$.  Therefore by Lemma~\ref{lema:R(z):W},
$g\gamma\cdot I=g\delta \cdot I$.  Hence 
\begin{equation*}
  \gamma_i\cdot I=\gamma\cdot I=\delta\cdot I=\delta_i\cdot I, \quad
  \text{for all large $i$},
\end{equation*}
a contradiction to \eqref{eq:gamma-delta}.
\end{proof}

\subsection{Growth properties of polynomial maps}

For any $\vv\in E$, the coordinate functions of the map $t\mapsto
u(t)\cdot\vv$ are polynomials of degree at most $2$. Therefore to
study the behaviour of the $U$-orbits on thin neighbourhoods of
compact subsets of $\cW$, we will use the growth properties of the
polynomial maps as described in the following basic observations
(see~\cite{Mar:non-div,Dani+Mar:limit}).

Let $l\geq 1$ be the dimension of $\qp$ over the topological closure
of $\Q$ in $\qp$. For a ball $B$ in $\qp$, let $\rad(B)$ denote the
radius of $B$ such that $\rad(B)=\abs{\lambda}$ for some
$\lambda\in\qp$. Then for any balls $B_1$ and $B_2$ in $\qp$, 
\begin{equation} \label{eq:l}
  \theta(B_2)=(r_2/r_1)^l\cdot\theta(B_1),\quad \text{where
  $r_i=\rad(B_i)$}. 
\end{equation} 

\begin{lema} \label{lema:poly:basic}
  Let $\epsilon>0$ and $d\in\N$ be given. Then there exists $c>0$ such
  that for any $f\in\qp[t]$
  with $\deg(f)\leq d$, and any ball $B$ in $\qp$,
  \begin{equation} \label{eq:poly:basic}
    \theta(\{t\in B: \abs{f(t)}<c\sup_{t\in B}\lvert f(t)\rvert\})\leq
    \epsilon \cdot \theta(B). 
  \end{equation}
  In fact, we can choose $c=C_d\inv (\epsilon/d)^{d/l}$, where $C_d=1$
  if $K$ is non-archimedean, and $C_d=(d+1)2^d$ if $K$ is archimedean.
\end{lema}

\begin{proof} 
  Put $M=\sup_{t\in B}\lvert f(t)\rvert$. Fix any $c>0$. Put $I=\{t\in
  B:\abs{f(t)}<c M\}$.  Suppose that
  \begin{equation} \label{eq:theta(I)}  
    \theta(I)>\epsilon\cdot\theta(B).
  \end{equation}
  
  We claim that there exist points $x_0,\ldots,x_d$ in $I$ such that
  \begin{equation} \label{eq:xi-xj}
    \abs{x_i-x_j}> (\epsilon/d)^{1/l}r,\quad \forall\, i\neq j,
  \end{equation}
  where $r$ denotes the radius of $B$. 
  
  To prove the claim, suppose that $x_0,\ldots,x_k$ are chosen so that
  (\ref{eq:xi-xj}) holds for $0\leq i,j\leq k$, where $0\leq k\leq
  d-1$.  Put
  \begin{equation*} 
    I'=\bigcup_{j=0}^k\, B_{x_j}((\epsilon/d)^{1/l}r). 
  \end{equation*}
  Then by \eqref{eq:l},
  \begin{equation*}
    \theta(I')\leq (k+1)(\epsilon/d)\theta(B)\leq \epsilon\theta(B).
  \end{equation*}
  By \eqref{eq:theta(I)} there exists $x_{i+1}\in I\setminus I'$. Then
  $\abs{x_{i+1}-x_j}\geq (\epsilon/d)^{1/l}$ for all $j\leq i-1$. This
  proves the claim.
  
  By Lagrange's interpolation formula,
  \begin{equation*}
    f(x)=\sum_{0\leq i\leq d} f(x_i)\prod_{j\neq i}\frac{x-x_j}{x_i-x_j}.
  \end{equation*}
  Now $\abs{f(x_i)}<cM$ and $\frac{\abs{x-x_j}}{\abs{x_i-x_j}}\leq
  2/(\epsilon/d)^{1/l}$ for all $j\neq i$, and $x\in\cl{B}$.
  Therefore $M<(d+1)cM 2^d/(\epsilon/d)^{d/l}$.  This leads to a
  contradiction if we choose $c=(1/(d+1)2^d)(\epsilon/d)^{d/l}$.
  Therefore \eqref{eq:theta(I)} cannot hold. 
\end{proof}

\begin{cor} \label{cor:grow}
  For any $f\in\qp[t]$ with $\deg(f)\leq d$, and balls $B_1\subset
  B_2$ in $\qp$, let $M_i=\sup_{t\in B_i} \abs{f(t)}$ and
  $r_i=\rad(B_i)$ for $i=1,2$. Then
  \begin{equation} \label{eq:grow}
    M_2 \leq C_d (r_2/r_1)^d M_1.
  \end{equation}
\end{cor} 

\begin{proof} 
  Let $0<\epsilon<(r_2/r_1)^{-l}$. Let $F=\{t\in B_2:\abs{f(t)}<
  C_d\inv\epsilon^{d/l} M_2\}$. Then by Lemma~\ref{lema:poly:basic}
  and \eqref{eq:l},
  \begin{equation*}
  \theta(F) \leq \epsilon \theta(B_2)=\epsilon (r_2/r_1)^l\theta(B_1)
  <\theta(B_x(r)).
  \end{equation*}
  Thus $B_x(r)\not\subset F$, and hence $M_1\geq C_d\inv \epsilon^{d/l}
  M_2$. Hence $M_2\geq C_d\lambda^d M_1$.
\end{proof}

\begin{prop} \label{prop:avoid}
  Given $\epsilon>0$ and a compact set $C\subset \cW$, there exists a
  compact set $D\subset \cW$ containing $C$ such that the following
  holds: given any neighbourhood $\Phi$ of $D$ in $\cW$ there exists a
  neighbourhood $\Psi$ of $C$ in $E$ such that for any $\vv\in E$ and
  any ball $\frakB$ in $\qp$, one of the following holds: 
\begin{gather} 
u(\frakB)\vv\subset\Phi \label{eq:in-Phi} \\
\intertext{or}
\theta(\{t\in\frakB: u(t)\vv\in\Psi\})\leq 
    \epsilon\cdot\theta(\{t\in\frakB:u(t)\vv\in \Phi\}). 
\label{eq:avoid} 
\end{gather}
\end{prop} 

\begin{proof} 
  Let $\{\phi_1,\ldots,\phi_4\}$ be linear functionals on $E$ such
  that
  \begin{equation*}
    \vy=
    \begin{pmatrix}
      \phi_1(\vy) & \phi_2(\vy) \\
      \phi_3(\vy) & \phi_4(\vy)
    \end{pmatrix}, \ \forall \vy\in E.
  \end{equation*}
  Then
  \begin{equation*}
    \cW=\{\vy\in E: \phi_i(\vy-I)=0,\ \forall\,i\neq 2\}.
  \end{equation*}
  Note that $\phi_2(\vy-I)=\phi_2(\vy)$ for all $\vy\in E$.  Define
  $f_i(t) = \phi_i(u(t)\vv-I)$ for all $i$ and $t\in\qp$.  Then
  $f_i\in\qp[t]$ and $\deg(f_i)\leq 2$.
  
  There exists $\alpha_2>0$ such that
  \begin{equation*}
    C\subset \{\vy\in \cW: \abs{\phi_2(\vy-I)}<\alpha_2\}
  \end{equation*}
  
  We fix a small $0<c<1$, whose value will be specified below.
  Let $M_2=c\inv\alpha_2$ and put
  \begin{equation*}
    D=\{\vy\in\cW: \abs{\phi_2(\vy-I)}\leq M_2\}.
  \end{equation*}
  
  Now given any neighbourhood $\Phi$ of $D$, there exists $M_i>0$ for
  each $i\neq 2$, such that
  \begin{equation*}
    \Phi\supset \{\vy\in E: \abs{\phi_i(\vy-I)}\leq M_i,\ \forall\,i\}.
  \end{equation*}
  We choose $\alpha_i=cM_i$ for each $i\neq 2$, and put
  \begin{equation*}
    \Psi = \{\vy\in E:\abs{\phi_i(\vy-I)}<\alpha_i,\ \forall\,i\}.
  \end{equation*} 
  Then $\Psi$ is a neighbourhood of $C$.
  
  Define
  \begin{align}
    F   & = \{t\in\frakB:\abs{f_i(t)} < M_i,\ \forall\, i\} 
    \label{eq:F}\\
    & \subset \{t\in\frakB:u(t)\vv\in \Phi\} ,
    \label{eq:F-2} \\
    \intertext{and}
    F_1 & = \{t\in\frakB:\abs{f_i(t)}< \alpha_i,\ \forall\,i\} 
    \label{eq:F1} \\
    & = \{t\in\frakB:u(t)\vv\in\Psi\}.
    \label{eq:F1-2}
  \end{align}
  Suppose that \eqref{eq:in-Phi} does not hold. Then
  \begin{equation*}
    \frakB\not\subset F.
  \end{equation*}
  
  A ball $B\subset F$ is a called a {\em maximal ball\/} in $F$, if
  $B'\not\subset F$ for any ball $B'\subset\frakB$ strictly bigger
  than $B$.
  
  Let $B$ be a maximal ball in $F$. We claim that
  \begin{equation}
    \label{eq:maxB} 
    \sup_{t\in B}\abs{f_{i_0}(t)}\geq \tau\inv M_{i_0},\quad \text{for
      some $i_0$},
  \end{equation}
  where $\tau=p^2C_2>1$ if $\qp$ is non-archimedean, and $\tau=1$ if
  $\qp$ is archimedean.

  Suppose if $\sup_{t\in B} \abs{f_i(t)}<M_i$ for all $i$, then
  $B\subset F\subsetneq\frakB$. Then there exists a ball
  $B'\subset\frakB$ strictly bigger than $B$. Hence $B'\not\subset F$.
  Therefore by \eqref{eq:F}, for some $i_0$,
  \begin{equation*}
    \sup_{t\in B'}\abs{f_{i_0}(t)}\geq M_{i_0}.
  \end{equation*}
  If $\qp$ is a finite extension of $\Q_p$, we choose $B'$ such that
  $\rad(B')/\rad(B)=p$; and \eqref{eq:grow} implies \eqref{eq:maxB}.
  If $\qp$ is archimedean, \eqref{eq:maxB} is straightforward to
  conclude.

  If $\qp$ is non-archimedean or $\qp=\R$, then any two intersecting
  maximal balls in $F$ are same. Therefore $F=\cup\cB$, where $\cB$
  denotes the collection of disjoint maximal balls of $F$.  If
  $\qp=\C$ then there exists a collection $\cB'$ of disjoint maximal
  balls in $F$ such that if we put $\cB=\{B_x(3r): B_x(r)\in \cB'\}$
  then $F\subset\cup\cB$ (cf.~\cite[Proof of
  Lemma~8.4]{Rudin:RealComplex}).  Therefore
  \begin{equation} \label{eq:sum:B}
    \sum_{B\in\cB} \theta(B)\leq \kappa \theta(F)
  \end{equation}
  where $\kappa=1$ if $\qp\neq \C$ and $\kappa=9$ if $\qp=\C$. 
  
  We specify the value $c=(C_2\epsilon^{2/l})\inv \tau\inv$.  Let
  $B\in\cB$. Therefore by \eqref{eq:maxB}, there exists $i_0$ such
  that 
\begin{equation*}
\sup_{t\in B} \abs{f_{i_0}(t)}\geq \tau\inv M_{i_0}.
\end{equation*} 
Since $\alpha_{i_0}=cM_{i_0}$, by \eqref{eq:F1} and
Lemma~\ref{lema:poly:basic} applied to $f_{i_0}$:
  \begin{equation*} 
   \theta(F_1\cap B)\leq \theta(\{t\in B:\abs{f_{i_0}(t)}
   <\alpha_{i_0}\})\leq \epsilon\cdot\theta(B).
  \end{equation*}
  
  Therefore by \eqref{eq:sum:B}, we get that
  \begin{align}
    \theta(F_1)& =\theta(F_1\cap(\cup\cB)) \nonumber\\
    & \leq \sum_{B\in\cup\cB} \theta(F_1\cap B) 
    \leq \epsilon \sum_{B\in\cup\cB} \theta(B) 
    \leq \kappa \epsilon\cdot \theta(F). \nonumber
  \end{align}
  Therefore \eqref{eq:avoid} follows from \eqref{eq:F-2} and
  \eqref{eq:F1-2}.
\end{proof}

\subsection{Proof of Theorem~\ref{thm:avoid:2}:}

Given $\epsilon>0$ and a compact set $C_1\subset W$, put $C=C_1\cdot I
\subset \cW$, and obtain $D\subset\cW$ as in
Proposition~\ref{prop:avoid}. By Lemma~\ref{lema:inject}, there exists
a neighbourhood $\Phi$ of $D$ in $E$ such that
\begin{equation} \label{eq:card1}
  \card{\cR(z)\cap\Phi}\leq 1,\quad \forall\, z\in G/\Gamma.
\end{equation}
In other words, every element of $G/\Gamma$ can have at most one
representative in $\Phi$.

The set $W_D:=\{w\in W_1:w\cdot I\in D\}$ is compact. Also $\{g:g\cdot
I\in D\}=W_DH$. Now $D_1:=W_DHy$ is a compact subset of $WHy\subset
G/\Gamma$. Since $x\not\in WHy$, and $\cR(x)\cap D=\emptyset$. Since
$\cR(x)$ is discrete and $D$ is compact, there exists a compact
neighbourhood $V$ of the identity in $G$ such that $V\cR(x)\cap
D=\emptyset$. We replace $\Phi$ by $\Phi\setminus V\cR(x)$, which is
an open neighbourhood of $D$. Since $x_i\to x$, we have $x_i\in
V\cR(x)$ for all $i\geq i_0$ for some $i_0$. Since $\cR(vx)=v\cR(x)$
for all $v\in V$, we have that have that
\begin{equation} \label{eq:not-in-Phi}
\cR(x_i)\cap \Phi=\emptyset, \quad \forall\,i\geq i_0.
\end{equation}

By Proposition~\ref{prop:avoid} and \eqref{eq:not-in-Phi}, there
exists a neighbourhood $\Psi$ of $C$ in $E$ contained in $\Phi$ such
that for any $T>0$, 
\begin{equation} \label{eq:Psi-Phi}
\begin{split}
  &\theta(\{t\in B_0(T) :  u(t)\vv\in\Psi\})\\
&\leq  \epsilon\cdot\theta(\{t\in B_0(T):u(t)\vv\in\Phi\}), \quad \forall\,
  \vv\not\in \Phi.
\end{split}
\end{equation}

Let $\Psi_1=\{gy:g\cdot I\in\Psi,\, g\in G\}\subset G/\Gamma$. Since
$\Psi$ is a neighbourhood of $C=C_1H\cdot I$ in $E$, we conclude that
$\Psi_1$ is a neighbourhood of $C_1Hy$ in $G/\Gamma$.

Now fix $T>0$. For a subset $\Omega\subset E$,
define
\begin{equation*}
  L_\Omega(\vv)=\{t\in B_0(T): u(t)\vv\in \Omega\},\quad \forall\,
  \vv\in E.
\end{equation*}

We observe that 
\begin{equation}
\label{eq:Psi_1-Psi}
\{t\in B_0(T):u(t)x_i\in \Psi_1\}=\bigcup_{\vv\in\cR(x_i)}\,
L_\Psi(\vv).
\end{equation}

By \eqref{eq:not-in-Phi} and \eqref{eq:Psi-Phi}, 
\begin{equation} \label{eq:compare}
\theta(L_\Psi(\vv))\leq \epsilon\cdot \theta(L_\Phi(\vv)), \quad
\forall\,\vv\in\cR(x_i).
\end{equation}

We claim that
\begin{equation} \label{eq:disjoint}
L_\Phi(\vv_1)\cap L_\Phi(\vv_2)=\emptyset,\quad \forall\, \vv_1\neq
\vv_2,\,\vv_i\in\cR(x_i).
\end{equation}

If the claim is false, then there exists $t\in L_\Phi(\vv_1)\cap
L_\Phi(\vv_2)$. Therefore $\{u(t)\vv_1,u(t)\vv_2\}\subset
\cR(u(t)x_i)\cap\Phi$ and $u(t)\vv_1\neq u(t)\vv_2$. This contradicts
\eqref{eq:card1}. This proves the claim.

Now due to (\ref{eq:disjoint}),
\begin{equation} \label{eq:disj}
\sum_{\vv\in \cR(x_i)} \theta(L_\Phi(\vv))\leq \theta(t_i\zp).
\end{equation}

Combining (\ref{eq:Psi_1-Psi}), (\ref{eq:compare}), and
(\ref{eq:disj}), we get 
\begin{equation*}
\theta(\{t\in B_0(T):u(t)x\in\Psi_1\})\leq
\epsilon\cdot\Theta(t_i\zp), \quad \forall\, i\geq i_0.
\end{equation*} 
\qed

\section{Uniform distribution for unipotent orbits}

\subsection{Proof of Lemma~\ref{lema:inv}} 
Let $\epsilon>0$ be given.  Since $0<\theta(\frakO)<\infty$ by the
observation as before, we may assume that $\frakO$ is compact. Now
since $\theta$ is translation invariant and regular, there exists
$\delta>0$, such that for any $s\in\qp$ with $\abs{s}\leq \delta$, we
have
\begin{equation*}
  \theta((\frakO+s)\,\Delta\,\frakO)/\theta(\frakO)\geq \epsilon,
\end{equation*}
where $A\,\Delta\,B:=(A\setminus B)\cup(B\setminus A)$.

Let $s\in\qp$ be given. If $t\in\qp$ such that $\abs{t}\geq
\delta\inv\abs{s}$, then
\begin{align*}
  \theta((t\frakO+s)\,\Delta\, t\frakO)/\theta(t\frakO) 
  & = \theta(t((\frakO+t\inv s)\,\Delta\,\frakO))/\theta(t\frakO) \\
  & = \theta((\frakO+t\inv s)\,\Delta\,\frakO)/\theta(\frakO)\leq\epsilon.
\end{align*}
Let $i_0\in\N$ be such that $\abs{t_i}\geq \delta\inv\abs{s}$ for
all $i\geq i_0$.  Then for any Borel set $E\subset G/\Gamma$,
\begin{equation*}
  \abs{\mu_i(u(-s)E)-\mu_i(E)}\leq \theta((t_i\frakO+s)\,\Delta\,
  \frakO)/\theta(t_i\frakO)\leq\epsilon,\quad \forall\,i\geq i_0.
\end{equation*} 
Therefore $\abs{\mu(u(-s)E)-\mu(E)}\leq\epsilon$. Since $\epsilon$, $s$,
and $E$ are arbitrary, $\mu$ is $U$-invariant. 
\qed

\subsection{On the definition of singular set}

We begin with a group theoretic observation. 

\begin{prop} 
  \label{prop:wHw} 
  Suppose $F$ is a closed subgroup of $G$
  containing $U$ and $x\in G/\Gamma$ such that $Fx$ is compact. Then
  there exists $\cJ\in\cC$ and $w\in W$ such that $wH_{\cJ}w\inv
  \subset F \subset Z(G)(wH_{\cJ}w\inv)$.
\end{prop}

\begin{proof}
  First we consider the case of $n=1$, that is $G=\SL_2(\qp)$. Now
  suppose that $F\subset N_G(U)=DU$. Since $[F,F]\subset U$, by
  Proposition~\ref{prop:cocomp}, $[F_x,F_x]\subset U\cap G_x=\{e\}$.
  Therefore $F_x$ is an abelian subgroup of $DU$. Also since $F_x\cap
  U=\{e\}$, it is straightforward to verify that $F_x\subset uDu\inv$
  for some $u\in U$. Since $F=U(uDu\inv\cap F)$, it follows that
  $F/F_x$ cannot be compact, a contradiction.

Therefore there exists $f\in F$ such that $U':=fUf\inv\neq U$. Then
for the standard $\SL_2(\qp)$ action on $\qp^2$,
\begin{equation*}
  UU'\smat{1 \\ 0}
  = \qp^2 \setminus \smat{\qp \\ 0}.
\end{equation*}  
Hence
\begin{equation*}
  U'UU'\smat{1\\ 0}=\qp^2\setminus\{0\}.
\end{equation*}
Since the stabilizer of $\smat{1 \\ 0}$ is $U$, we have that
$U'UU'U=\SL_2(\qp)$. Therefore $F=G=H_{\{1\}}$, and the proof is
complete.

We intend to prove the general case by induction on $n$. Therefore we
assume that the proposition is valid for $k$ in place of $n$, where
$k=1,\ldots,n-1$.

For $i\in\{1,\ldots,n\}$ and $J=\{1,\ldots,n\}\setminus\{i\}$, let
$p_i:G\to G_J$ and $\bar p_i:G/\Gamma \to G_J/\Gamma_J$ be the natural
quotient maps.

Now if $F\supset G_i$ for some $i$, then $p_i(F)\bar p_i(x)=\bar
p_i(Fx)$ is compact. Since $p_i(U)$ plays the role of $U$ in $G_J$,
the general result easily follows from the induction hypothesis.

Now we assume that $F\not\supset G_j$ for each $j$.  For any
$i=1,\ldots,n$, let $q_i=G\to G_i$ and $\bar q_i:G/\Gamma\to
G_{\{i\}}/\Gamma_i$ be the natural projection maps for $i=1,\ldots,n$.
Then from the case of $n=1$ we deduce that
\begin{equation*}
  \bar q_i(Fx)=q_i(F)\bar q_i(x)=G_i/\Gamma_i.  
\end{equation*}
Hence $q_i(F)=G_i$.  

Let $F_1=G_{\{1\}}F$. Since $G_{\{1\}}y$ is compact for all $y\in
G/\Gamma$, we have that $F_1x$ is compact. Therefore by what we have
proved above there exists $w\in W$ and $\cJ\in\cC$ such that
\begin{equation*}
  wH_{\cJ}w\inv \subset F_1 \subset Z(G) wH_{\cJ}w\inv.
\end{equation*}

If $H_\cJ\neq G$ then $wH_{\cJ}w\inv\cong \SL_2(\qp)^k$ for some
$1\leq k\leq n-1$. Since
\begin{equation*}
  wH_{\cJ}w\inv/(wH_{\cJ}w\inv)_x = 
  \prod_{J\in\cJ} wH_J w\inv/(wH_Jw\inv)_x,
\end{equation*}
we conclude the result from the induction hypothesis.  

Therefore we can assume that $G_{\{1\}}F=F_1=G$. Since
$F\cap\ker(q_1)=F\cap G_{\{2,\ldots,n\}}$ is a normal subgroup of $F$,
and it commutes with $G_{\{1\}}$, we have that $F\cap\ker(q_1)$ is
normal in $G$ and in particular it is a normal subgroup of
$G_{\{2,\ldots,n\}}$.  Since we have assumed that $F$ does not contain
$G_{\{j\}}$ for any $j$, we have that $F\cap\ker(q_1)\subset Z(G)$.
Since $q_1(F)=G_{\{1\}}$, we have $n=2$. Thus $G=\SL_2(\qp)\times
\SL_2(\qp)$, and $\Lie(F)\cong \Lie(\SL_2(\qp))$.  By the same
argument as above $F\cap\ker(q_2)\subset Z(G)$. Since projection of
$F$ on each of the factors is surjective, there exists $g\in G_2$ such
that
\begin{equation*}
  \Lie(F)=\{(X,\Ad(g)X):X\in\Lie(\SL_2(\qp))\}
\end{equation*}
Since $U\subset F$, we have that
$\Ad(g)\smat{0&1\\0&0}=\smat{0&1\\0&0}$. Therefore $g\in Z(G)W\cap
G_2$. Therefore we can choose $w\in W$ such that
$wH_{\{1,2\}}w\inv\subset F \subset Z(G)wH_{\{1,2\}}w\inv$. This
proves the proposition in all the cases.
\end{proof}

Now from the above result it is straightforward to deduce the
following:

\begin{cor} 
\label{cor:singular}
The singular set $\cS(U,\Gamma)$ consists of those $x\in G/\Gamma$
such that $Fx$ is compact for some proper closed subgroup $F$ of $G$
containing $U$. \qed
\end{cor}

\subsection{Ergodic $U$-invariant measures on $G/\Gamma$}

The following description of $U$-ergodic measures was obtained in
\cite{R:measure,R:p-adic,Mar+Tom:Invent}

\begin{theo}[Ratner, Margulis-Tomanov]
  \label{thm:ergodic}
  Let $\lambda$ be a $U$-invariant $U$-ergodic probability measure on
  $G/\Gamma$. Then there exists a closed subgroup $F$ of $G$
  containing $U$ and a point $x\in G/\Gamma$ such that $Fx\cong F/F_x$
  is compact and $\lambda$ is the unique $F$-invariant probability
  measure supported on $Fx$.

  In particular, by Corollary~\ref{cor:singular}, if
  $\lambda(\cS(U,\Gamma))=0$ then $\lambda$ is $G$-invariant.  \qed
\end{theo}

\subsection{Proof of Theorem~\ref{thm:uniform}.}
\label{sec:proof-uniform}
We intend to prove this result by induction on $n$.

If $x\in \cS(U,\Gamma)$ then there exists $w\in W$ and $\cJ\in\cC$
such that if we put $F=wH_{\cJ}w\inv$ then $Ux\subset Fx$, $Fx$ is
compact, and $F\cong \SL_2(\qp)^k$, where $k=\abs{\cJ}\leq n-1$. Since
\begin{equation*}
  Fx\cong F/F_x=
  \prod_{J\in\cJ} wH_Jw\inv/(wH_Jw\inv)_x 
\end{equation*}
and $wH_Jw\inv\cong \SL_2(\qp)$, we can replace $G$ by $F$ and the
result follows from the induction hypothesis.

Therefore now we can assume that $x\in G/\Gamma\setminus
\cS(U,\Gamma)$. We put $x_i=x$ for all $i$. Choose any sequence
$T_i\to\infty$ in $\qp$. Then by (\ref{eq:mu_i}) and (\ref{eq:mu_Ti})
we have that $\mu_i=\mu_{T_i}$ for all $i$. Now by passing to a
subsequence we may assume that $\mu_i\to \mu$ for some $\mu\in\cM$;
that is, for any $f\in C(G/\Gamma)$,
\begin{equation*}
  \lim_{i\to\infty} \int_{G/\Gamma} f\, d\mu_i = 
  \int_{G/\Gamma} f \, d\mu.
\end{equation*}

By Lemma~\ref{lema:inv} we have that $\mu$ is $U$-invariant.  By
Theorem~\ref{thm:avoid} we have that $\mu(\cS(U,\Gamma))=0$. Therefore
in view of the decomposition of an invariant measure into its ergodic
components, we have that $\lambda(\cS(U,\Gamma))=0$ for almost all
$U$-ergodic components $\lambda$ of $\mu$. Therefore by
Theorem~\ref{thm:ergodic} almost all $U$-ergodic components of $\mu$
are $G$-invariant. Hence $\mu$ is $G$-invariant.  \qed

\end{document}